\documentclass[11pt]{amsart}

\usepackage{cite}
\usepackage[pagebackref=true, colorlinks=true, citecolor=blue]{hyperref}

\usepackage{amsmath,latexsym}

\author[Ambrose]{David M. Ambrose$^{*}$}
\address{Department of Mathematics, Drexel University, Philadelphia, PA 19104, USA}
\email{dma68@drexel.edu}
\address{$^{*}$ Corresponding author.}

\author[Lopes Filho]{Milton C. Lopes Filho}
\address{Instituto de Matematica, Universidade Federal do Rio de Janeiro, Caixa Postal 68530,
Rio de Janeiro, RJ, 21941-909 Brazil}
\email{mlopes@im.ufrj.br}

\author[Mazzucato]{Anna L. Mazzucato}
\address{Department of Mathematics, Pennsylvania State University, University Park, PA 16802, USA}
\email{alm24@psu.edu}

\author[Nussenzveig Lopes]{Helena J. Nussenzveig Lopes}
\address{Instituto de Matematica, Universidade Federal do Rio de Janeiro, Caixa Postal 68530,
Rio de Janeiro, RJ, 21941-909 Brazil}
\email{hlopes@im.ufrj.br}

\title[Pseudomeasure distributions for mean field games]{Pseudomeasure distributions for nonseparable, nonlocal mean field games}

\theoremstyle{definition}
\newtheorem{example}{Example}
\newtheorem{theorem}{Theorem}
\newtheorem{lemma}{Lemma}
\newtheorem{assumption}{Assumption}
\newtheorem{remark}{Remark}

\begin{document}

\begin{abstract}For a number of important mean field games models, the Hamiltonian is non-local and not additively
separable.  This means that the distribution of agents appears in the Hamiltonian only in an integral over the
whole spatial domain.  For mean field games with a class of such Hamiltonians, we prove existence of solutions
for the mean field games system of partial differential equations, allowing pseudomeasure data for the distribution
of agents.  Specifically, this allows the initial distribution of agents to be a sum of Dirac masses.  The existence
theorem requires a smallness condition on the size of the terminal data for the value function (or, alternatively, on the
size of the Hamiltonian); no smallness condition
on the size of the initial data or on the size of the time horizon is required.  We also prove uniqueness and continuous 
dependence results under the same type of smallness conditions.  We prove continuous dependence under two complementary 
hypotheses on the initial data:
strong convergence of a sequence of pseudomeasures, and weak-$*$ convergence
of a sequence of bounded measures.
\end{abstract}



\maketitle

\section{Introduction}

The mean field games system of partial differential equations is
\begin{equation}\label{originalUEquation}
u_{t}+H(x,t,\nabla u, m)=-\Delta u,
\end{equation}
\begin{equation}\label{originalMEquation}
m_{t}+\mathrm{div}(mH_{p}(x,t,\nabla u,m)=\Delta m.
\end{equation}
The nonlinearity is known as the Hamiltonian and is denoted by $H.$  When $H$ is regarded as a function $H=H(x,t,p,m),$
the function $H_{p}$ is the gradient of $H$ with respect to the $p$-variables.
We consider the problem in $d$ spatial dimensions; we specifically will use $\mathbb{T}^{d}$ as the spatial domain,
with $\mathbb{T}^{d}$ being the hypercube $[0,2\pi]^{d}$ with periodic boundary conditions.
The temporal domain is $[0,T]$ for some given $T>0.$
The system \eqref{originalUEquation}, \eqref{originalMEquation} is taken together with initial and terminal conditions,
\begin{equation}\label{originalTerminalCondition}
u(\cdot,T)=G(m(\cdot,T),\cdot),
\end{equation}
\begin{equation}\label{originalInitialCondition}
m(0,\cdot)=m_{0}.
\end{equation}
Here $G$ is known as the payoff function, although it will typically be a nonlocal operator rather than a local function.
As only $\nabla u$ influences the evolution of $u$ and $m,$ we will change variables to $v=\nabla u.$
Then the system becomes
\begin{equation}\nonumber
v_{t}+\nabla(H(x,t,v,m))=-\Delta v,
\end{equation}
\begin{equation}\nonumber
m_{t}+\mathrm{div}(mH_{p}(x,t,v,m))=\Delta m.
\end{equation}
The terminal condition \eqref{originalTerminalCondition} becomes simply
\begin{equation}\label{terminalCondition}
v(\cdot,T)=\nabla G(m(\cdot,T),\cdot).
\end{equation}

Many important works in the analysis of mean field games (such as \cite{gomesLog}, \cite{gomesSuper}, \cite{gomesSub},
\cite{lasryLions1}, \cite{lasryLions2}, \cite{lasryLions3},
 among others) assume that the Hamiltonian is additively separable into a part depending on $v$ and
a part depending on $m.$  However, in applications
the Hamiltonian might not not have this additively separable property.
For example, the model of Chan and Sircar of natural resource extraction \cite{chanSircar} has Hamiltonian
\begin{equation}\nonumber
H=\frac{1}{4}\left(\frac{2}{2+\varepsilon\eta(t)}+\frac{\varepsilon}{2+\varepsilon\eta(t)}\int_{0}^{L}u_{x}\ \mathrm{d}m - u_{x}\right)^{2}.
\end{equation}
Here, the spatial domain is the interval $[0,L],$ $\varepsilon$ is a parameter related to substitutibility of commodities,
and $\eta(t)$ is the total measure in the problem at time $t.$
Existence of solutions for the Chan-Sircar model has been studied by Bensoussan and Graber \cite{jameson1}, and Graber and Mouzouni
extended the existence theory to the case of measure-valued data \cite{jameson2}.
Another example of an applied mean field games problem with nonseparable Hamiltonian is the model of household savings and wealth
introduced by Achdou, Buera, Lasry, Lions, and Moll
\cite{macroeconomic}, \cite{macroeconomic2}.
In this model, the Hamilton-Jacobi equation for the value function being optimized and the Fokker-Planck equation for the distribution of agents
are coupled only by the
interest rate, which is a function only of time.  The formula for the interest rate as developed by the first author in \cite{AMOP} (which is the only study of the
time-dependent model in the literature) involves the measure variable only in integrals over the entire spatial domain, and is multiplied in the
integrand by functions of the unknown value function.

In these important applied examples, the measure $m$ appears in the Hamiltonian only in an integral over the entire spatial domain.
In the present work, we consider nonseparable Hamiltonians of this form, specifically
\begin{equation}\label{prototypeHamiltonian}
H(x,t,v,m)=g(v)\int_{\mathbb{T}^{d}}f(v)m\ \mathrm{d}x.
\end{equation}
For simplicity, as a prototype of such a problem, we consider the system to be
\begin{equation}\label{finalVEquation}
v_{t}+\nabla(H_{1}(x,t,v,m))=-\Delta v,
\end{equation}
\begin{equation}\label{finalMEquation}
m_{t}+\mathrm{div}(mH_{2}(x,t,v,m))=\Delta m,
\end{equation}
where $H_{1}$ and $H_{2}$ are given by
\begin{equation}\label{formOfHamiltonians}
H_{i}(x,t,v,m)=g_{i}(v)\int_{\mathbb{T}^{d}}f_{i}(v)m\ \mathrm{d}x.
\end{equation}
(We note that this is more general than the mean field games system, and that to recover the mean field games system $H_{2}$ should be 
taken to be $H_{1p}.$)
Naturally, some conditions will be imposed on the functions $f_{i}$ and $g_{i}.$ Since $v$ is vector-valued, $H_{1}$
is scalar-valued and $H_{2}$ is vector-valued.  We therefore take $f_{1}$ and $g_{1}$ to be scalar-valued, and we have a choice
regarding $f_{2}$ and $g_{2}.$ To be definite, we take $f_{2}$ to be scalar-valued, and $g_{2}$ to be vector-valued, but all of our results
are also valid with the opposite choice, or also when taking the sum of the two choices.  We also remark that we could take the functions
$f_{i}$ and $g_{i}$ to depend on the independent variables $x$ and $t$ in addition to depending on $v,$ as long as the $f_{i}$ and $g_{i}$
maintain the desired mapping properties (to be developed in Section \ref{assumptionsSection} below); for simplicity, however, we leave them depending only
on $v.$ We fix the terminal time $T>0$ and impose the initial and terminal conditions \eqref{originalInitialCondition}, \eqref{terminalCondition}.

The first author previously proved existence of solutions for the mean field games system in the case of separable local Hamiltonians with local coupling
and nonseparable local Hamiltonians using a contraction mapping argument in function spaces based on the Wiener algebra \cite{CRAS2}, \cite{JMPA};
the method of these papers is closely related to the work of Duchon and Robert on vortex sheets in incompressible fluid dynamics \cite{duchonRobert}.
In the present work we will continue to use the Wiener algebra as the functional setting for the value function $u,$ but we will now be using
the space of pseudomeasures for the distribution of agents, $m.$
The set of pseudomeasures is the set of functions with Fourier series in the space $\ell^{\infty}.$  It is immediate to see, then, that any bounded
measure on $\mathbb{T}^{d}$ is a pseudomeasure, as all of the Fourier coefficients are bounded by the total measure.  In particular, Dirac masses
and sums of Dirac masses are available as initial data for our existence theory.
Function spaces based on pseudomeasures have been used for problems
in incompressible fluid dynamics and related nonlinear problems  \cite{ALN5}, \cite{ALN7}, \cite{ALN6}, \cite{cannoneKarch}.
The precise function spaces to be used will be defined below in Section \ref{functionSpacesOperatorEstimates}.  The present work again uses the framework related to
the Duchon-Robert work, as previously used for mean field games by the first author in \cite{CRAS2}, \cite{JMPA}.

Of course, mean field games are meant to approximate $N$-player games with a large number of similar agents.  This connection has been
made rigorous in particular cases, in which it is shown that the $N\rightarrow\infty$ limit of $N$-player games can be taken \cite{delarueEtAl}.
This was done with a separable Hamiltonian with structural assumptions such as convexity and monotonicity of the resulting parts of the Hamiltonian,
and under a growth condition that the Hamiltonian acts similarly to $|\nabla u|^{2}.$  Naturally, it is of interest to be able to rigorously establish
the $N\rightarrow\infty$ limit for additional Hamiltonians.  The general plan of the proof of \cite{delarueEtAl} is to use solutions of the mean field
games system of partial differential equations to find solutions of the master equation, and to then use these solutions of the master equation
to control the solutions of the $N$-player games as $N$ increases to infinity.  For the mean field games system of partial
differential equations, to provide suitable solutions of the master equation for this purpose, the appropriate initial data must be taken.  The
relevant initial data for the $N$-player game is a normalized sum of $N$ Dirac masses.  Thus, in proving the existence of solutions for
\eqref{finalVEquation}, \eqref{finalMEquation} with initial data $m_{0}$ in the set of pseudomeasures, we carry out the first step in a program of
eventually establishing the $N\rightarrow\infty$ limit for the associated $N$-player games.

Existence theory for the mean field games system of partial differential equations with nonseparable Hamiltonians has been provided in prior works both
for particular applied problems, for systems with particular structure, or for general classes of systems under smallness conditions.
For particular applied problems, we have already mentioned work on the Chan-Sircar model
of natural resource extraction and the Achdou-Buera-Lasry-Lions-Moll model of household wealth.  Additional applications featuring non-separable
Hamiltonians are congestion problems, as in \cite{congestion1}, \cite{congestion2}.
Well-posedness has also been established under the additional structural assumption of displacement monotonicity \cite{displacement1}.
Existence of solutions for general classes
of nonseparable Hamitlonians, under various assumptions on the data, have been proved by the first author \cite{JMPA}, \cite{IUMJ4}, and
by Cirant, Gianni, and Mannucci \cite{cirantEtAl}.  In all these prior works dealing with non-separable Hamiltonians, the data is always taken
to be sufficiently regular so as to rule out the use of Dirac mass data, except in \cite{jameson2} and \cite{displacement1}.
The present work is the first to prove existence
of solutions for mean field games with measure data for a class of nonseparable Hamiltonians rather than for a specific application, and without
a structural assumption such as displacement monotonicity. Our main results are a local-in-time wellposedness result in adapted spaces based on the Wiener algebra for $u$ and on the space of pseudomeasures for $m$ (Theorem  \ref{Main1}), and continuity results for the solution with respect to the data for $m$, both in the pseudo-measure norm (Theorem \ref{Main2}), as well as with respect to weak-$\ast$ convergence of measures  (Theorem \ref{Main3}).

The plan of the paper is as follows.  In Section \ref{functionSpacesOperatorEstimates}, we define our function spaces and
develop the mild formulation of the problem.  In Section \ref{assumptionsSection}, we state the assumptions we will need on the Hamiltonian and
on the payoff function; we also give examples which satisfy these assumptions.  In Section \ref{existenceSection} we then apply the contraction mapping
theorem to prove existence of mild solutions to the mean field games system.  In Section \ref{convergenceSection}, we prove continuous dependence of
solutions on the initial data, in two ways.  Finally, Appendix \ref{appendix} contains some estimates for operators on our function spaces.

\section{Function spaces}\label{functionSpacesOperatorEstimates}

We will be proving an existence theorem with initial data for $m$ in the space of pseudomeasures, $PM^{0};$
this is the set of periodic functions with bounded Fourier coefficients.
More generally, we will use the spaces $PM^{\beta}$ for $\beta\geq0,$ which we will now define.
For a given $\alpha\geq0,$ we also define a space-time version of this, $\mathcal{PM}^{\alpha}.$
We will define both spaces through their norms.  We let a periodic function $f$ have Fourier series representation
\begin{equation}\nonumber
f(x)=\sum_{k\in\mathbb{Z}^{d}}f_{k}e^{ikx}.
\end{equation}
Of course, if the function $f$ also depends on time, then the Fourier coefficients $f_{k}$ satisfy $f_{k}=f_{k}(t).$
Given $\beta\geq0,$ the space $PM^{\beta}$ is defined through its norm as
\begin{equation}\nonumber
\|f\|_{PM^{\beta}}=\sup_{k\in\mathbb{Z}^{d}}e^{\beta |k|}|f_{k}|,
\end{equation}
and given $\alpha\geq0,$ the space of functions on space-time, $\mathcal{PM}^{\alpha},$ is defined through its norm as
\begin{equation}\nonumber
\|f\|_{\mathcal{PM}^{\alpha}}=\sup_{k\in\mathbb{Z}^{d}}\sup_{t\in[0,T]}e^{\alpha t|k|}|f_{k}(t)|.
\end{equation}

In our existence theory, the derivative of the value function, $v,$ will be in function spaces based on the Wiener algebra;
the Wiener algebra is the space of functions with Fourier series in $\ell^{1}.$  We again introduce versions of these spaces
with exponential weights.  For $\beta\geq0,$ we introduce the purely spatial norm for $B_{\beta},$
and a space-time version $\mathcal{B}_{\beta}.$ The norm for $B_{\beta}$ is
\begin{equation}\nonumber
\|f\|_{B_{\beta}}=\sum_{k\in\mathbb{Z}^{d}}e^{\beta |k|}|f_{k}|,
\end{equation}
and the norm for the related function space on space-time, $\mathcal{B}_{\beta},$ is
\begin{equation}\nonumber
\|f\|_{\mathcal{B}_{\beta}}=\sum_{k\in\mathbb{Z}^{d}}\sup_{t\in[0,T]}e^{\beta|k|}|f_{k}(t)|.
\end{equation}
Notice that for $\mathcal{PM}^{\alpha},$ we have taken the exponential weight to be time-dependent, while for $\mathcal{B}_{\beta},$ we have
taken a fixed exponential weight.  This is a choice we have made, and other choices are possible and reasonable.  We will comment more on this below.

We will need one further related space, which is the space of functions which have one derivative in $B_{\beta},$ for a given $\beta\geq0.$
We call this space $B_{1,\beta},$ and
its norm is
\begin{equation}\nonumber
\|f\|_{B_{1,\beta}}=\sum_{k\in\mathbb{Z}^{d}}(1+|k|)e^{\beta |k|}|f_{k}|.
\end{equation}
The Wiener algebra, $B_{0},$ is (of course) a Banach algebra, and this property is inherited by all the spaces
$B_{\beta},$ $\mathcal{B}_{\beta}$ and $B_{1,\beta}.$
That is, for any of these spaces, we have $\|fg\|\leq\|f\|\|g\|,$ where all of the norms in the inequality are taken in the same space.

\begin{remark}
The initial data, $m_{0},$ will be taken to be in the space $PM^{0}$ in our existence theorem.  As discussed in the introduction, this
is because it is of interest in the theory of mean field games to allow the initial data to be both a probability measure and a sum of Dirac masses.  
That $m_{0}$ could be a probability measure and the sum of Dirac masses is allowed by our use of the space of pseudomeasures for initial data; 
we will not always remark that we take $m_{0}$ to satisfy these conditions.
Furthermore, the equation
\eqref{finalMEquation} is parabolic, and solutions of parabolic equations gain regularity at positive times.
We will be proving the existence of $m\in\mathcal{PM}^{\alpha},$
for some $\alpha>0,$ and the gain of regularity is reflected in this choice of space.  At each time $t\in[0,T],$ we will have $m(\cdot,t)\in PM^{\alpha t}.$
Therefore, at each $t\in(0,T],$ the Fourier coefficients of the solution $m(\cdot,t)$ decay exponentially, with exponential decay rate at least $\alpha t.$
This implies that at time $t,$ the distribution $m(\cdot,t)$ is analytic with radius of analyticity at least $\alpha t.$  As remarked in the series of papers
\cite{ALN5}, \cite{ALN4}, \cite{ALN7}, \cite{ALN6}, it would be expected that this could be improved at short times to a faster rate, proportional to $\sqrt{t}.$
However, since we do not take $T$ small, the linear rate of growth of the radius of analyticity we prove is more useful.

For $v,$ since it satisfies a backward parabolic equation, we give its terminal data at time $t=T.$  In the mean field games problem, this terminal data
is in terms of the payoff function $G$ applied to the final distribution, $m(\cdot,T).$  This $m(\cdot,T),$ then, is analytic with radius of analyticity at least $\alpha T.$
We will be assuming that $G$ maps to functions which are also analytic with radius of analyticity at least $\alpha T;$  see
Assumption \ref{payoffAssumption} in Section \ref{assumptionsSection} below for specific details.
We then choose to prove that $v$ maintains radius of analyticity at least $\alpha T$ throughout the entire time interval $[0,T].$
On the one hand, we could prove that $v$ gains further analyticity owing the fact that $v$ also satisfies a parabolic equation, \eqref{finalVEquation};
we choose not to pursue this
for simplicity, as the additional gain of regularity will not be of further benefit to us.  On the other hand, we note that maintaining this level of regularity throughout
the time interval is only possible because of the nonlocal structure of the Hamiltonian \eqref{formOfHamiltonians}.  If $v_{t}$ were to depend locally on $m,$
we would not expect $v(\cdot,t)$ to be significantly smoother than $m(\cdot,t).$  By integrating $m$ over the entire spatial domain,
$H_{1}$ is not limited by the regularity of $m.$  Thus we are in a situation in which at all times $t$ in $[0,T),$ $v(\cdot,t)$ will be more regular than $m(\cdot,t).$
\end{remark}

We note that if $\beta\geq\alpha T,$ then the product of a $\mathcal{PM}^{\alpha}$ function and a $\mathcal{B}_{\beta}$ function is in
$\mathcal{PM}^{\alpha}.$  Indeed, we have the estimate
\begin{multline}\label{productEstimate}
\|fg\|_{\mathcal{PM}^{\alpha}}=\sup_{k\in\mathbb{Z}^{d}}\sup_{t\in[0,T]}e^{\alpha t|k|}
\left|\sum_{j\in\mathbb{Z}^{d}}f_{k-j}(t)g_{j}(t)\right|
\\
\leq \sup_{k\in\mathbb{Z}^{d}}\sum_{j\in\mathbb{Z}^{d}}\left(\sup_{t\in[0,T]}e^{\alpha t|k-j|}|f_{k-j}(t)|\right)
\left(\sup_{j\in\mathbb{Z}^{d}}\sup_{t\in[0,T]}e^{\alpha t|j|}|g_{j}(t)|\right)
\\
\leq \sup_{k\in\mathbb{Z}^{d}}\sum_{j\in\mathbb{Z}^{d}}\left(\sup_{t\in[0,T]}e^{\beta|k-j|}|f_{k-j}(t)|\right)\|g\|_{\mathcal{PM}^{\alpha}}
=\|f\|_{\mathcal{B}_{\beta}}\|g\|_{\mathcal{PM}^{\alpha}}.
\end{multline}

With $v(\cdot,T)\in (B_{\alpha T})^{d}$ and $m_{0}\in PM^{0},$ we will be proving an existence theorem for
$v\in(\mathcal{B}_{\alpha T})^{d}$ and $m\in\mathcal{PM}^{\alpha}.$

We write the Duhamel formulation for $v$ and $m.$  For $m$ this is the familiar formula, but for $v$ we must remember that
we integrate backward in time from time $T.$  We have
\begin{equation}\label{duhamelVOriginal}
v(t)=e^{{\Delta}(T-t)}v(\cdot,T)-\int_{t}^{T}e^{{\Delta}(s-t)}\nabla(H_{1}(v,m))(s)\ \mathrm{d}s,
\end{equation}
\begin{equation}\label{duhamelMOriginal}
m(\cdot,t)=e^{{\Delta}t}m_{0}+\int_{0}^{t}e^{{\Delta}(t-s)}\mathrm{div}(mH_{2}(v,m))(s)\ \mathrm{d}s.
\end{equation}
We define the operators $I^{+}:\mathcal({B}_{\alpha T})^{d}\times\mathcal{PM}^{\alpha}\rightarrow\mathcal{PM}^{\alpha}$
and $I^{-}:\mathcal{B}_{\alpha T}\rightarrow(\mathcal{B}_{\alpha T})^{d}$ to be
\begin{equation}\nonumber
I^{+}(f,g)(\cdot,t)=\int_{0}^{t}e^{{\Delta}(t-s)}\mathrm{div}(f(\cdot,s)g(\cdot,s))\ \mathrm{d}s,
\end{equation}
\begin{equation}\nonumber
I^{-}(h)(\cdot,t)=-\int_{t}^{T}e^{{\Delta}(s-t)}\nabla(h(\cdot,s))\ \mathrm{ds}.
\end{equation}
We prove mapping properties for these operators in Appendix \ref{appendix}.

Using the definition of $I^{+},$ the Duhamel formula for $m$ becomes
\begin{equation}\label{duhamelM}
m(\cdot,t)=e^{\Delta t}m_{0}+I^{+}(m,H_{2}(v,m))(t).
\end{equation}
Before arriving at the corresponding formula for $v,$ we must express $v(\cdot,T)$ using the payoff function $G$ and
the terminal distribution $m(\cdot,T).$
To this end, we introduce the notation $\Omega(v,m)$ to indicate the terminal distribution $m(\cdot,T),$ as it depends on $v$ and
$m$ through the Duhamel formula.  Using \eqref{duhamelM} with $t=T,$ we have
\begin{equation}\label{omegaDefinition}
\Omega(v,m)=e^{\Delta T}m_{0}+I^{+}(m,H_{2}(v,m))(\cdot,T).
\end{equation}
Using the definitions of $I^{-}$ and $\Omega,$ and using the terminal condition \eqref{terminalCondition},
we finally express the Duhamel formula for $v$ as
\begin{equation}\label{duhamelV}
v(\cdot,t)=e^{\Delta(T-t)}(\nabla G(\Omega(v,m),\cdot))+I^{-}(H_{1}(v,m))(\cdot,t).
\end{equation}
The main effort of the present work will be to show the existence of a solution $(v,m)$ of the system \eqref{duhamelM}, \eqref{duhamelV}.

\section{Assumptions on the Hamiltonian and Payoff Function}\label{assumptionsSection}

We now state assumptions on the functions $f_{i},$ $g_{i}$ which were introduced as constituent parts of the Hamiltonian functions $H_{1}$ and $H_{2}$ in
\eqref{formOfHamiltonians}.
\begin{assumption}\label{HamiltonianAssumption} For any $T>0$ and for any $\alpha\in[0,1),$
the functions $f_{1},$ $f_{2},$ and $g_{1}$ map $(\mathcal{B}_{\alpha T})^{d}$ to $\mathcal{B}_{\alpha T}$
and the function $g_{2}$ maps $(\mathcal{B}_{\alpha T})^{d}$ to
$(\mathcal{B}_{\alpha T})^{d}.$  For all $h\in\{f_{1},f_{2},g_{1},g_{2}\},$ there
exists $c_{h}>0$ and $p_{h}\geq0$ such that for all $v\in(\mathcal{B}_{\alpha T})^{d},$ 
\begin{equation}\label{assumptionOnFunctions}
\|h(v)\|\leq c_{h}\|v\|_{(\mathcal{B}_{\alpha T})^{d}}^{p_{h}},
\end{equation}
where the norm on the left-hand side is either $\mathcal{B}_{\alpha T}$ or $(\mathcal{B}_{\alpha T})^{d},$ as appropriate.
Furthermore, for all $h\in\{f_{1},f_{2},g_{1},g_{2}\},$ there
exists $\tilde{c}_{h}>0$ and $\tilde{p}_{h}\geq0$ such that for all $v_{1}\in(\mathcal{B}_{\alpha T})^{d}$ and $v_{2}\in(\mathcal{B}_{\alpha T})^{d},$ 
\begin{equation}\nonumber
\|h(v_{1})-h(v_{2})\|\leq \tilde{c}_{h}\|v_{1}-v_{2}\|_{(\mathcal{B}_{\alpha T})^{d}}
\left(\|v_{1}\|_{(\mathcal{B}_{\alpha T})^{d}}^{\tilde{p}_{h}}+\|v_{2}\|_{(\mathcal{B}_{\alpha} T)^{d}}^{\tilde{p}_{h}}\right),
\end{equation}
where again the norm on the left-hand side is either $\mathcal{B}_{\alpha T}$ or $(\mathcal{B}_{\alpha T})^{d},$ as appropriate.
\end{assumption}

In Appendix \ref{appendix}, we give a bound for the integrals $A_{i}:=\int_{\mathbb{T}^{d}}f_{i}(v)m\ \mathrm{d}x$ which appear in the formula
\eqref{formOfHamiltonians} defining $H_{1}$ and $H_{2}.$
Together with this bound, \eqref{ABound}, Assumption \ref{HamiltonianAssumption} implies a bound on $H_{1}$ and $H_{2},$ namely
\begin{equation}\label{H1Bound}
\|H_{1}(v,m)\|_{\mathcal{B}_{\alpha T}}\leq c_{f_{1}}c_{g_{1}}\|v\|_{(\mathcal{B}_{\alpha T})^{d}}^{p_{f_{1}}+p_{g_{1}}}\|m\|_{\mathcal{PM}^{\alpha}},
\end{equation}
\begin{equation}\label{H2Bound}
\|H_{2}(v,m)\|_{(\mathcal{B}_{\alpha T})^{d}}\leq c_{f_{2}}c_{g_{2}}\|v\|_{(\mathcal{B}_{\alpha T})^{d}}^{p_{f_{2}}+p_{g_{2}}}\|m\|_{\mathcal{PM}^{\alpha}}.
\end{equation}

\begin{example}\label{firstExample}
We take $H_{1}$ and $H_{2}$ to be given by
\begin{equation}\nonumber
H_{1}(v,m)=|v|^{2}\int_{\mathbb{T}^{d}}|v|^{2}m\ \mathrm{d}x,
\end{equation}
\begin{equation}\nonumber
H_{2}(v,m)=2v\int_{\mathbb{T}^{d}}|v|^{2}m\ \mathrm{d}x.
\end{equation}
Then we have $g_{1}(v)=f_{1}(v)=f_{2}(v)=|v|^{2},$ and $g_{2}(v)=2v.$
We have the estimates
\begin{equation}\nonumber
\||v|^{2}\|_{\mathcal{B}_{\alpha}}\leq \|v\|_{\mathcal{B}_{\alpha}^{d}}^{2}
\end{equation}
and
\begin{equation}\nonumber
\||v_{1}|^{2}-|v_{2}|^{2}\|_{\mathcal{B}_{\alpha}}\leq \|v_{1}-v_{2}\|_{\mathcal{B}_{\alpha}^{d}}(\|v_{1}\|_{\mathcal{B}_{\alpha}^{d}}+\|v_{2}\|_{\mathcal{B}_{\alpha}^{d}}).
\end{equation}
Thus we see that 
we can take $c_{f_{1}}=c_{g_{1}}=c_{f_{2}}=1,$ and $p_{f_{1}}=p_{g_{1}}=p_{f_{2}}=2.$
We also have $c_{g_{2}}=2$ and $p_{g_{2}}=1.$ For the Lipschitz estimates, we have $\tilde{c}_{f_{1}}=\tilde{c}_{g_{1}}=\tilde{c}_{f_{2}}=1$ and
$\tilde{p}_{f_{1}}=\tilde{p}_{g_{1}}=\tilde{p}_{f_{2}}=1.$ For $g_{2},$ we get $\tilde{c}_{g_{2}}=1$ and $\tilde{p}_{g_{2}}=0.$
\end{example}

\begin{example} We can also do an example which is quadratic in $v.$ Say we took the simple example in which
$H_{1}(v,m)=|v|^{2}$ and $H_{2}(v,m)=2v.$  Then we would have $g_{1}(v)=|v|^{2},$ $g_{2}(v)=2v$ and $f_{1}(v)=f_{2}(v)=1.$
Then $c_{f_{1}}=c_{f_{2}}=1$ and $p_{f_{1}}=p_{f_{2}}=0,$ and $\tilde{c}_{f_{1}}=\tilde{c}_{f_{2}}=0,$ and we could take $\tilde{p}_{f_{1}}=\tilde{p}_{f_{2}}=0$
as well.  Also, $c_{g_{1}},$ $p_{g_{1}},$ $\tilde{c}_{g_{1}},$ and $\tilde{p}_{g_{1}}$ are as in Example \ref{firstExample}.  Finally,
we would have $c_{g_{2}}=2$ and $p_{g_{2}}=1,$ and $\tilde{c}_{g_{2}}=1$ and $\tilde{p}_{g_{2}}=0.$
\end{example}

\begin{remark} We comment on the nature of our assumptions, and on our examples.  Our methods can treat more general classes of
Hamiltonians, including linear combinations of terms of the form specified by our assumptions.  For instance, our method of proof
extends to examples such as
\begin{equation}\nonumber
H_{1}(v,m)=|v|^{4}+(|v|^{2}+1)\int_{\mathbb{T}^{d}}|v|^{2}m\ \mathrm{d}x + \int_{\mathbb{T}^{d}}|v|^{4}m\ \mathrm{d}x,
\end{equation}
\begin{equation}\nonumber
H_{2}(v,m)=4|v|^{2}v+2v\int_{\mathbb{T}^{d}}|v|^{2}m\ \mathrm{d}x + 2(|v|^{2}+1)\int_{\mathbb{T}^{d}}vm\ \mathrm{d}x
+4\int_{\mathbb{T}^{d}}v|v|^{2}m\ \mathrm{d}x.
\end{equation}
The different terms in the two sums include different powers from each other, and also include different structures (such as in $H_{2}$
whether the vector-valued function is outside the integral or inside the integral).  While our methods fully apply to this example of $H_{1}$
and $H_{2},$ presenting the proof in this level of generality would only introduce further complications of exposition and notation.
Therefore we continue with our assumptions, in which our proofs apply directly to most of the individual terms in these sums.
Finally, we mention that while some works in the literature (specifically those dealing with the $N\rightarrow\infty$ problem) are limited
to Hamiltonians which are essentially quadratic in $v$ \cite{displacement2}, \cite{delarueEtAl}, there is no such requirement in the present work.
\end{remark}

We next state an assumed bound and a Lipschitz assumption on the payoff function, $G.$ Note that we take the payoff function to be smoothing;
this is a common assumption, e.g. the payoff function was taken to be smoothing for several of the results in \cite{IUMJ4} as well as the results of
\cite{cirantEtAl}.  As discussed in \cite{cirantEtAl}, the use of non-smoothing payoff functions requires additional smallness conditions, and such
an additional smallness condition is developed in \cite{IUMJ4}.  The work \cite{ambroseGriffin-PickeringMeszaros} on non-separable kinetic mean field games,
for instance, was completed under the assumption of non-smoothing payoff functions instead, with more restrictive smallness conditions.
\begin{assumption}\label{payoffAssumption}
The payoff function $G$ maps $PM^{\alpha T}\times\mathbb{T}^{d}$ to $B_{1,\alpha T}.$  There exist non-decreasing, continuous functions
$\Psi_{1}$ and $\Psi_{2}$ (mapping from $\mathbb{R}_{+}$ to $\mathbb{R}_{+}$), and
there exists constants $c_{G}>0$ and $\tilde{c}_{G}>0$ such that
for all $m,$ $m_{1},$ and $m_{2}$ in $PM^{\alpha T},$ we have the estimates
\begin{equation}\nonumber
\|G(m,\cdot)\|_{B_{1,\alpha T}}\leq c_{G}\|m\|_{PM^{\alpha T}}\Psi_{1}(\|m\|_{PM^{\alpha T}}),
\end{equation}
\begin{equation}\nonumber
\|G(m_{1},\cdot)-G(m_{2},\cdot)\|_{B_{1,\alpha T}}\leq \tilde{c}_{G}\|m_{1}-m_{2}\|_{PM^{\alpha T}}\Psi_{2}(\|m_{1}\|_{PM^{\alpha T}}+\|m_{2}\|_{PM^{\alpha T}}).
\end{equation}
\end{assumption}

\begin{example} If we let $\chi_{n}$ be the operator which truncates a Fourier series by zeroing out modes of size larger than $n,$ then
\begin{equation}\nonumber
G(m,x)=(\chi_{n}m)^{2}\sin(x_{1}+\ldots+x_{d})
\end{equation}
satisifes Assumption \ref{payoffAssumption}.  Notice that this $G(m,x)$ is compactly supported in Fourier space, and the mapping properties
(i.e. the smoothing)
stem from this fact.
\end{example}

\begin{example} If $m(\cdot,T)\in PM^{\alpha T},$ then we know that $\sup_{k\in\mathbb{Z}^{d}}e^{\alpha T |k|}|m_{k}(T)|$ is finite.
For any $\gamma>0,$ we define an operator $L^{\gamma}$ through its symbol as $\widehat{L^{\gamma}}(k)=\frac{1}{1+|k|^{1+d+\gamma}}.$  Then
we define $G(m(\cdot,T),\cdot)=L^{\gamma}m(\cdot,T).$  Then we have
\begin{multline}\nonumber
\|G(m(\cdot,T),\cdot)\|_{B_{1,\alpha T}}=\sum_{k\in\mathbb{Z}^{d}}\frac{1+|k|}{1+|k|^{1+d+\gamma}}e^{\alpha T|k|}|m_{k}(T)|
\\
\leq \|m(\cdot,T)\|_{PM^{\alpha T}} \sum_{k\in\mathbb{Z}^{d}}\frac{1+|k|}{1+|k|^{1+d+\gamma}}\leq c_{\gamma}\|m(\cdot,T)\|_{PM^{\alpha T}}.
\end{multline}
\end{example}

\begin{lemma}\label{hamiltonianLipschitz}
Let $\rho_{1}>0$ and $\rho_{2}>0$ be given.
For any $\nu_{1}$ and $\nu_{2}$ in $(\mathcal{B}_{\alpha T})^{d}$ with $\|\nu_{i}\|_{(\mathcal{B}_{\alpha T})^{d}}\leq \rho_{1},$ and for
any $\mu_{1}$ and $\mu_{2}$ in $\mathcal{PM}^{\alpha}$ with $\|\mu_{i}\|_{\mathcal{PM}^{\alpha}}\leq \rho_{2},$ we have the
estimates
\begin{multline}\nonumber
\|H_{1}(\nu_{1},\mu_{1})-H_{1}(\nu_{2},\mu_{2})\|_{\mathcal{B}_{\alpha T}}
\\
\leq \left(2c_{f_{1}}\tilde{c}_{g_{1}}\rho_{1}^{p_{f_{1}}+\tilde{p}_{g_{1}}}\rho_{2}
+2c_{g_{1}}\tilde{c}_{f_{1}}\rho_{1}^{p_{g_{1}}+\tilde{p}_{f_{1}}}\rho_{2}
+c_{f_{1}}c_{g_{1}}\rho_{1}^{p_{f_{1}}+p_{g_{1}}}\right)\cdot
\\
\cdot
\left(\|\nu_{1}-\nu_{2}\|_{(\mathcal{B}_{\alpha T})^{d}}+\|\mu_{1}-\mu_{2}\|_{\mathcal{PM}^{\alpha}}\right),
\end{multline}
\begin{multline}\nonumber
\|H_{2}(\nu_{1},\mu_{1})-H_{2}(\nu_{2},\mu_{2})\|_{(\mathcal{B}_{\alpha T})^{d}}
\\
\leq \left(2c_{f_{2}}\tilde{c}_{g_{2}}\rho_{1}^{p_{f_{2}}+\tilde{p}_{g_{2}}}\rho_{2}
+2c_{g_{2}}\tilde{c}_{f_{2}}\rho_{1}^{p_{g_{2}}+\tilde{p}_{f_{2}}}\rho_{2}
+c_{f_{2}}c_{g_{2}}\rho_{1}^{p_{f_{2}}+p_{g_{2}}}\right)\cdot
\\
\cdot
\left(\|\nu_{1}-\nu_{2}\|_{(\mathcal{B}_{\alpha T})^{d}}+\|\mu_{1}-\mu_{2}\|_{\mathcal{PM}^{\alpha}}\right).
\end{multline}

\end{lemma}
\begin{proof}
We let $(\nu_{1},\mu_{1})$ and $(\nu_{2},\mu_{2})$ be given with norms bounded as in the statement of the lemma.
We begin by writing
\begin{equation}\nonumber
\|H_{1}(\nu_{1},\mu_{1})-H_{1}(\nu_{2},\mu_{2})\|_{\mathcal{B}_{\alpha T}}
=\|g_{1}(\nu_{1})A_{1}(\nu_{1},\mu_{1})-g_{1}(\nu_{2})A_{1}(\nu_{2},\mu_{2})\|_{\mathcal{B}_{\alpha T}}.
\end{equation}
We add and subtract and use the triangle inequality, finding
\begin{multline}\nonumber
\|H_{1}(\nu_{1},\mu_{1})-H_{1}(\nu_{2},\mu_{2})\|_{\mathcal{B}_{\alpha T}}
\\
\leq \|g_{1}(\nu_{1})A_{1}(\nu_{1},\mu_{1})-g_{1}(\nu_{2})A_{1}(\nu_{1},\mu_{1})\|_{\mathcal{B}_{\alpha T}}
\\
+ \|g_{1}(\nu_{2})A_{1}(\nu_{1},\mu_{1})-g_{1}(\nu_{2})A_{1}(\nu_{2},\mu_{1})\|_{\mathcal{B}_{\alpha T}}
\\
+\|g_{1}(\nu_{2})A_{1}(\nu_{2},\mu_{1})-g_{1}(\nu_{2})A_{1}(\nu_{2},\mu_{2})\|_{\mathcal{B}_{\alpha T}}
=I+II+III.
\end{multline}
We bound each of these in turn, beginning with $I.$  We estimate it as
\begin{multline}\nonumber
I\leq \|g_{1}(\nu_{1})-g_{1}(\nu_{2})\|_{\mathcal{B}_{\alpha T}}
\sup_{t\in[0,T]}|A_{1}(\nu_{1},\mu_{1})|
\\
\leq
\tilde{c}_{g_{1}}\|\nu_{1}-\nu_{2}\|_{(\mathcal{B}_{\alpha T})^{d}}
\left(\|\nu_{1}\|_{(\mathcal{B}_{\alpha T})^{d}}^{\tilde{p}_{g_{1}}}+\|\nu_{2}\|_{(\mathcal{B}_{\alpha T})^{d}}^{\tilde{p}_{g_{1}}}\right)
\|f_{1}(\nu_{1})\|_{\mathcal{B}_{\alpha T}}\|\mu_{1}\|_{\mathcal{PM}^{\alpha}}.
\end{multline}
This can then be further estimated as
\begin{equation}\nonumber
I\leq 2c_{f_{1}}\tilde{c}_{g_{1}}\rho_{1}^{p_{f_{1}}+\tilde{p}_{g_{1}}}\rho_{2}\|\nu_{1}-\nu_{2}\|_{(\mathcal{B}_{\alpha T})^{d}}.
\end{equation}

The term $II$ can be bounded as
\begin{equation}\nonumber
II\leq \|g_{1}(\nu_{2})\|_{\mathcal{B}_{\alpha T}}\sup_{t\in[0,T]}|A_{1}(\nu_{1},\mu_{1})-A_{1}(\nu_{2},\mu_{1})|.
\end{equation}
We note that we can write
\begin{equation}\nonumber
A_{1}(\nu_{1},\mu_{1})-A_{1}(\nu_{2},\mu_{1})=\int_{\mathbb{T}^{d}}(f_{1}(\nu_{1})-f_{1}(\nu_{2}))\mu_{1}\ \mathrm{d}x,
\end{equation}
and therefore we may bound
\begin{equation}\nonumber
\sup_{t\in[0,T]}|A_{1}(\nu_{1},\mu_{1})-A_{1}(\nu_{2},\mu_{1})|
\leq \|f_{1}(\nu_{1})-f_{1}(\nu_{2})\|_{\mathcal{B}_{\alpha T}}\|\mu_{1}\|_{\mathcal{PM}^{\alpha}}.
\end{equation}
Combining these bounds and our assumptions on $f_{1}$ and $g_{1},$ we estimate $II$ as
\begin{equation}\nonumber
II\leq 2c_{g_{1}}\tilde{c}_{f_{1}}\rho_{1}^{p_{g_{1}}+\tilde{p}_{f_{1}}}\rho_{2}\|\nu_{1}-\nu_{2}\|_{(\mathcal{B}_{\alpha T})^{d}}.
\end{equation}

The term $III$ can be bounded as (since $A_{1}$ is linear with respect to the $\mu$-variable)
\begin{multline}\nonumber
III\leq \|g_{1}(\nu_{2})\|_{\mathcal{B}_{\alpha T}}\left(\sup_{t\in[0,T]}|A_{1}(\nu_{2},\mu_{1}-\mu_{2})|\right)
\\
\leq \|g_{1}(\nu_{2})\|_{\mathcal{B}_{\alpha T}}\|f_{1}(\nu_{2})\|_{\mathcal{B}_{\alpha T}}
\|\mu_{1}-\mu_{2}\|_{\mathcal{PM}^{\alpha}}.
\end{multline}
Using our assumptions, this can then be bounded in turn as
\begin{equation}\nonumber
III\leq c_{f_{1}}c_{g_{1}}\rho_{1}^{p_{f_{1}}+p_{g_{1}}}\|\mu_{1}-\mu_{2}\|_{\mathcal{PM}^{\alpha}}.
\end{equation}
This completes the proof of the Lipschitz estimate for $H_{1}.$ The proof for $H_{2}$ is just the same, and we omit further details.
\end{proof}

\section{Existence theorem}\label{existenceSection}

We define a mapping $\mathcal{T}:(\mathcal{B}_{\alpha T})^{d}\times\mathcal{PM}^{\alpha}\rightarrow(\mathcal{B}_{\alpha T})^{d}\times\mathcal{PM}^{\alpha}$ as
$\mathcal{T}=(\mathcal{T}_{1},\mathcal{T}_{2}),$ with these components given by
\begin{equation}\nonumber
\mathcal{T}_{1}(\nu,\mu)=e^{{\Delta}(T-t)}\nabla G(\Omega(\nu,\mu),\cdot)+I^{-}(H_{1}(\nu,\mu)),
\end{equation}
\begin{equation}\nonumber
\mathcal{T}_{2}(\nu,\mu)=e^{{\Delta}t}m_{0}+I^{+}(\mu H_{2}(\nu,\mu)).
\end{equation}
Then we see that a mild solution of the mean field games system \eqref{finalVEquation}, \eqref{finalMEquation}, with data \eqref{originalInitialCondition},
\eqref{terminalCondition}, satisfies \eqref{duhamelM}, \eqref{duhamelV}, and is therefore a fixed point of $\mathcal{T},$
\begin{equation}\nonumber
(v,m)=\mathcal{T}(v,m).
\end{equation}
We will now prove existence of a fixed point by the contraction mapping theorem.

\begin{theorem} \label{Main1}
Let $T>0$ and let $\alpha\in[0,1)$ be given.  Let $H_{1}$ and $H_{2}$ be defined by \eqref{formOfHamiltonians}, and be such that
Assumption \ref{HamiltonianAssumption} is satisfied.  Let $G$ be such that Assumption \ref{payoffAssumption} is satisfied.

(a) If all of the following hold:
\begin{equation}\nonumber
p_{f_{1}}+p_{g_{1}}>1, \qquad p_{f_{2}}+p_{g_{2}}>1, \qquad p_{f_{1}}+\tilde{p}_{g_{1}}>0,
\end{equation}
\begin{equation}\nonumber
\tilde{p}_{f_{1}}+p_{g_{1}}>0, \qquad p_{f_{2}}+\tilde{p}_{g_{2}}>0, \qquad \tilde{p}_{f_{2}}+p_{g_{2}}>0,
\end{equation}
then there exists $\varepsilon>0$ such that if $c_{G}+\tilde{c}_{G}<\varepsilon,$ then
there exists a  mild solution of \eqref{finalVEquation}, \eqref{finalMEquation} with data \eqref{originalInitialCondition}, \eqref{terminalCondition},
and this mild solution is unique in a ball in $(\mathcal{B}_{\alpha T})^{d}\times\mathcal{PM}^{\alpha}.$\\

(b) there exists $\varepsilon>0$ such that if
\begin{equation}\nonumber
c_{f_{1}}c_{g_{1}}+c_{f_{2}}c_{g_{2}}+\tilde{c}_{f_{1}}c_{g_{1}}+c_{f_{1}}\tilde{c}_{g_{1}}+c_{f_{2}}\tilde{c}_{g_{2}}+\tilde{c}_{f_{2}}c_{g_{2}}<\varepsilon,
\end{equation}
then there exists a mild solution of \eqref{finalVEquation}, \eqref{finalMEquation} with data \eqref{originalInitialCondition}, \eqref{terminalCondition},
and this mild solution is unique in a ball in $(\mathcal{B}_{\alpha T})^{d}\times\mathcal{PM}^{\alpha}.$
\end{theorem}

\begin{remark} Of course, we must remark upon the conditions (a) and (b) in the theorem.  If the Hamiltonian $H_{1}$ is more than quadratic with respect to $v,$ then
in a realistic mean field games setting the corresponding function $H_{2}$ (which would really be the derivative $\nabla_{p}H_{1}$) would be more than linear with respect to $v.$
When performing the contraction estimate below, we get differences of $H_{2}.$ When $H_{2}$ is superlinear, say $H_{2}\sim v^{1+\delta}$ with $\delta>0,$ then
$H_{2}(v_{1},\cdot)-H_{2}(v_{2},\cdot) \sim (v_{1}^{\delta}+v_{2}^{\delta})|v_{1}-v_{2}|.$  Then, if $v$ can be taken to be small, we get the required contracting property.
We can make $v$ small by taking the payoff function $G$ to be small, and this is reflected in the assumption in (a) that $c_{G}$ and $\tilde{c}_{G}$ are small.
Said another way, if we were to replace $G$ by $\delta G,$ we could get the required properties by taking $\delta$ sufficiently small.
This is the reason for the assumption (a) above.

If instead we have that $H_{2}$ is linear, say $H_{2}(v,\cdot)\sim v,$ then when making a difference estimate we get no additional smallness from taking $v$ small.
In this case, to satisfy the contracting estimate, we need the Hamiltonian itself to be small (i.e. if $H_{2}$ itself came with a small constant in front of it, then
this will give the contracting property).  In fact, no matter what, if the Hamiltonian itself is small, then we can get the needed properties.  This is the reason for
the assumption (b) above.
\end{remark}

\begin{proof}
We let $R_{0}$ denote $R_{0}=\|m_{0}\|_{PM^{0}}.$ We let $R_{1}$ denote
\begin{equation}\nonumber
R_{1}=\|e^{{\Delta}(T-t)}\nabla G(e^{\Delta T}m_{0},\cdot)\|_{(\mathcal{B}_{\alpha T})^{d}}.
\end{equation}
We define $X$ to be the ball in $(\mathcal{B}_{\alpha T})^{d}\times\mathcal{PM}^{\alpha}$ 
\begin{equation}\nonumber
(e^{{\Delta}(T-t)}\nabla G(e^{\Delta T}m_{0},\cdot),e^{{\Delta}t}m_{0}),
\end{equation}
with radius $R_{1}.$  We will show that, in either case (a) or (b) of the theorem, with the relevant constants chosen sufficiently small, we have that $\mathcal{T}$ maps
$X$ to $X$ and  that $\mathcal{T}$ is a contraction.  For future reference, we note that Assumption \ref{payoffAssumption} implies the bound
\begin{equation}\label{R1Bound}
R_{1}\leq \|G(e^{\Delta T}m_{0},\cdot)\|_{B_{1,\alpha T}} \leq c_{G}R_{0}\Psi_{1}(R_{0}),
\end{equation}
and thus for fixed $m_{0},$ if $c_{G}$ is small then $R_{1}$ is small as well.

To show that $\mathcal{T}$ maps $X$ to $X,$ we let $(\nu,\mu)\in X,$ and we remark that $\|\nu\|_{(\mathcal{B}_{\alpha T})^{d}}\leq 2R_{1}$ and
$\|\mu\|_{\mathcal{PM}^{\alpha}}\leq R_{0}+R_{1}.$ We first must estimate
\begin{equation}\nonumber
\|\mathcal{T}_{1}(\nu,\mu)-e^{{\Delta}(T-t)}\nabla G(e^{\Delta T}m_{0},\cdot)\|_{(\mathcal{B}_{\alpha T})^{d}}.
\end{equation}
Using the definition of $\mathcal{T}_{1}$ and the triangle inequality, we can estimate this as
\begin{multline}\label{T1MapsXToX}
\|\mathcal{T}_{1}(\nu,\mu)-e^{\Delta(T-t)}\nabla G(e^{\Delta T}m_{0},\cdot)\|_{(\mathcal{B}_{\alpha T})^{d}}
\\
\leq
\|e^{\Delta(T-t)}\nabla\left[G(\Omega(\nu,\mu),\cdot)-G(e^{\Delta T}m_{0},\cdot)\right]\|_{(\mathcal{B}_{\alpha T})^{d}}
+\|I^{-}(H_{1}(\nu,\mu))\|_{(\mathcal{B}_{\alpha T})^{d}}.
\end{multline}
For the first term on the right-hand side, elementary estimates allow us to bound it as
\begin{multline}\nonumber
\left\|e^{\Delta(T-t)}\nabla\left[G(\Omega(\nu,\mu),\cdot)-G(e^{\Delta T}m_{0},\cdot)\right]\right\|_{(\mathcal{B}_{\alpha T})^{d}}
\\
\leq
\left\|G(\Omega(\nu,\mu),\cdot)-G(e^{\Delta T}m_{0},\cdot)\right\|_{B_{1,\alpha T}}.
\end{multline}
We then use Assumption \ref{payoffAssumption} and \eqref{omegaDefinition}, which is the definition of $\Omega,$ to find
\begin{multline}\label{beforeBreakingForOmega}
\left\|e^{\Delta(T-t)}\nabla\left[G(\Omega(\nu,\mu),\cdot)-G(e^{\Delta T}m_{0},\cdot)\right]\right\|_{(\mathcal{B}_{\alpha T})^{d}}
\\
\leq \tilde{c}_{G}\| I^{+}(\mu,H_{2}(\nu,\mu))(\cdot,T)\|_{PM^{\alpha T}}\Psi_{2}(\|\Omega(\nu,\mu)\|_{PM^{\alpha T}}+\|e^{\Delta T}m_{0}\|_{PM^{\alpha T}}).
\end{multline}

We now focus on bounding the last factor on the right-hand side of \eqref{beforeBreakingForOmega}.
We have $\left\|e^{\Delta T}m_{0}\right\|_{PM^{\alpha T}}\leq \|m_{0}\|_{PM^{0}}=R_{0}.$
Again using the definition of $\Omega,$ and also using Lemma \ref{I+Lemma}, we can bound the norm of $\Omega$ as
\begin{equation}\nonumber
\|\Omega(\nu,\mu)\|_{PM^{\alpha T}}\leq\|m_{0}\|_{PM^{0}}+\frac{1}{1-\alpha}\|\mu\|_{\mathcal{PM}^{\alpha}}\|H_{2}(\nu,\mu)\|_{(\mathcal{B}_{\alpha T})^{d}}.
\end{equation}
Then, using \eqref{H2Bound} and \eqref{R1Bound}, this becomes
\begin{multline}\nonumber
\|\Omega(\nu,\mu)\|_{PM^{\alpha T}}\leq R_{0}+\frac{1}{1-\alpha}(R_{0}+R_{1})^{2}c_{f_{2}}c_{g_{2}}(2R_{1})^{p_{f_{1}}+p_{f_{2}}}
\\
\leq
R_{0}+\frac{1}{1-\alpha}(R_{0}+c_{G}R_{0}\Psi_{1}(R_{0}))^{2}c_{f_{2}}c_{g_{2}}(2c_{G}R_{0}\Psi_{1}(R_{0}))^{p_{f_{2}}+p_{g_{2}}}.
\end{multline}
We define
\begin{equation}\label{UpsilonDefinition}
\Upsilon=\Psi_{2}\left(2R_{0}+\frac{1}{1-\alpha}(R_{0}+c_{G}R_{0}\Psi_{1}(R_{0}))^{2}c_{f_{2}}c_{g_{2}}(2c_{G}R_{0}\Psi_{1}(R_{0}))^{p_{f_{2}}+p_{g_{2}}}
\right),
\end{equation}
and we conclude that
\begin{equation}\nonumber
\Psi_{2}\left(\|\Omega(\nu,\mu)\|_{PM^{\alpha T}}+\left\|e^{\Delta T}m_{0}\right\|_{PM^{\alpha T}}\right) \leq \Upsilon.
\end{equation}

Returning to \eqref{beforeBreakingForOmega}, the definition of $\mathcal{PM}^{\alpha}$ then implies
\begin{equation}\nonumber
\left\|e^{\Delta(T-t)}\nabla\left[G(\Omega(\nu,\mu),\cdot)-G(e^{\Delta T}m_{0},\cdot)\right]\right\|_{(\mathcal{B}_{\alpha T})^{d}}
\leq \tilde{c}_{G}\Upsilon\| I^{+}(\mu,H_{2}(\nu,\mu))\|_{\mathcal{PM}^{\alpha}}.
\end{equation}
Boundedness of $I^{+}$ (see Lemma \ref{I+Lemma} in Appendix \ref{appendix}) then implies
\begin{multline}\nonumber
\left\|e^{\Delta(T-t)}\nabla\left[G(\Omega(\nu,\mu),\cdot)-G(e^{\Delta T}m_{0},\cdot)\right]\right\|_{(\mathcal{B}_{\alpha T})^{d}}
\\
\leq \tilde{c}_{G}\Upsilon\left(\frac{1}{1-\alpha}\right)\|\mu\|_{\mathcal{PM}^{\alpha}}\|H_{2}(\nu,\mu)\|_{(\mathcal{B}_{\alpha T})^{d}}.
\end{multline}
We then use \eqref{H2Bound}, finding
\begin{multline}\nonumber
\left\|e^{\Delta(T-t)}\nabla\left[G(\Omega(\nu,\mu),\cdot)-G(e^{\Delta T}m_{0},\cdot)\right]\right\|_{(\mathcal{B}_{\alpha T})^{d}}
\\
\leq \tilde{c}_{G}c_{g_{2}}c_{f_{2}}\Upsilon\left(\frac{1}{1-\alpha}\right)
\|\mu\|_{\mathcal{PM}^{\alpha}}^{2}\|\nu\|_{(\mathcal{B}_{\alpha T})^{d}}^{p_{g_{2}}+p_{f_{2}}}.
\end{multline}
We can then bound this as
\begin{multline}\label{finalFirstTermT1XToX}
\left\|e^{\Delta(T-t)}\nabla\left[G(\Omega(\nu,\mu),\cdot)-G(e^{\Delta T}m_{0},\cdot)\right]\right\|_{(\mathcal{B}_{\alpha T})^{d}}
\\
\leq
\tilde{c}_{G}c_{g_{2}}c_{f_{2}}\Upsilon\left(\frac{1}{1-\alpha}\right) (2R_{1})^{p_{g_{2}}+p_{f_{2}}}(R_{0}+R_{1})^{2}.
\end{multline}
We see that the right-hand side of \eqref{finalFirstTermT1XToX} can be made smaller than $R_{1}/4$ by either taking
$\tilde{c}_{G}$ to be sufficiently small (in case (a) of the theorem), or otherwise by taking $c_{f_{2}}c_{g_{2}}$ to be sufficiently small.  We have established
\begin{equation}\label{finalfinalFirstTermT1XToX}
\left\|e^{\Delta(T-t)}\nabla\left[G(\Omega(\nu,\mu),\cdot)-G(e^{\Delta T}m_{0},\cdot)\right]\right\|_{(\mathcal{B}_{\alpha T})^{d}}
\leq \frac{R_{1}}{4}.
\end{equation}

We next consider the second term on the right-hand side of \eqref{T1MapsXToX}.  
By Lemma \ref{I-Lemma} of Appendix \ref{appendix}, we have the bound
\begin{equation}\nonumber
\|I^{-}(H_{1}(\nu,\mu))\|_{(\mathcal{B}_{\alpha T})^{d}}\leq d\|H_{1}(\nu,\mu)\|_{\mathcal{B}_{\alpha T}}.
\end{equation}
We bound $H_{1}$ using \eqref{H1Bound}, namely
\begin{equation}\nonumber
\|I^{-}(H_{1}(\nu,\mu))\|_{(\mathcal{B}_{\alpha T})^{d}}\leq dc_{f_{1}}c_{g_{1}}\|\nu\|_{(\mathcal{B}_{\alpha T})^{d}}^{p_{g_{1}}+p_{f_{1}}}
\|\mu\|_{\mathcal{PM}^{\alpha}}.
\end{equation}
We can then bound this in turn as
\begin{equation}\label{finalSecondTermT1XToX}
\|I^{-}(H_{1}(\nu,\mu))\|_{(\mathcal{B}_{\alpha T})^{d}}\leq dc_{f_{1}}c_{g_{1}}(2R_{1})^{p_{g_{1}}+p_{f_{1}}}
(R_{0}+R_{1}).
\end{equation}
In case (a) of the theorem, we have $p_{g_{1}}+p_{f_{1}}>1.$  If we take $c_{G}$ sufficiently small, then
$R_{1}^{p_{g_{1}}+p_{f_{1}}-1}$ becomes small.  Therefore, in this case, we can take $c_{G}$ sufficiently small so that
\begin{equation}\label{finalfinalSecondTermT1XToX}
\|I^{-}(H_{1}(\nu,\mu))\|_{(\mathcal{B}_{\alpha T})^{d}}\leq \frac{R_{1}}{4}.
\end{equation}
In the general case (i.e. in case (b) of the theorem), we can take $c_{f_{1}}c_{g_{1}}$ sufficiently small so that \eqref{finalfinalSecondTermT1XToX} is satisfied.

Combining \eqref{finalfinalFirstTermT1XToX} and \eqref{finalfinalSecondTermT1XToX}, we have established
\begin{equation}\label{XtoXFirstConclusion}
\left\|\mathcal{T}_{1}(\nu,\mu)-e^{\Delta(T-t)}\nabla G(e^{\Delta T}m_{0},\cdot)\right\|_{(\mathcal{B}_{\alpha T})^{d}}
\leq \frac{R_{1}}{2}.
\end{equation}

The corresponding estimates for $\mathcal{T}_{2}$ are simpler; we begin with
\begin{multline}\nonumber
\|\mathcal{T}_{2}(\nu,\mu)-e^{{\Delta}t}m_{0}\|_{\mathcal{PM}^{\alpha}}
=\|I^{+}(\mu H_{2}(\nu,\mu))\|_{\mathcal{PM}^{\alpha}}
\\
\leq\frac{1}{1-\alpha}\|\mu\|_{\mathcal{PM}^{\alpha}}\|H_{2}(\nu,\mu)\|_{(\mathcal{B}_{\alpha T})^{d}}.
\end{multline}
We then proceed as before, using \eqref{H2Bound} to bound $H_{2}$ and using the definitions of $R_{0}$ and $R_{1}.$ We then have the
following bound:
\begin{equation}\nonumber
\|\mathcal{T}_{2}(\nu,\mu)-e^{{\Delta}t}m_{0}\|_{\mathcal{PM}^{\alpha}}
\leq
\frac{1}{1-\alpha}
c_{f_{2}}c_{g_{2}}(2R_{1})^{p_{f_{2}}+p_{g_{2}}}(R_{0}+R_{1})^{2}.
\end{equation}
We therefore see
that if $c_{f_{2}}c_{g_{2}}R_{1}^{p_{f_{2}}+p_{g_{2}}-1}$ is sufficiently small, then
\begin{equation}\label{XtoXSecondConclusion}
\|\mathcal{T}_{2}(\nu,\mu)-e^{{\Delta}t}m_{0}\|_{\mathcal{PM}^{\alpha}}\leq \frac{R_{1}}{2}.
\end{equation}
As before, we can either have $c_{f_{2}}c_{g_{2}}$ small (this is case (b)), or in case (a) we have $p_{f_{2}}+p_{g_{2}}>1.$ In case (a), taking $c_{G}$ small will
have as a consequence that $R_{1}^{p_{f_{2}}+p_{g_{2}}-1}$ is also small.  Thus we have established \eqref{XtoXSecondConclusion}.
Combining \eqref{XtoXFirstConclusion} and \eqref{XtoXSecondConclusion}, we have shown that $\mathcal{T}$ maps $X$ to $X.$

We next consider the contracting property.  We let $(\nu_{1},\mu_{1})\in X$ and $(\nu_{2},\mu_{2})\in X$ be given.
We begin by estimating
\begin{multline}\label{T1Contracting}
\|\mathcal{T}_{1}(\nu_{1},\mu_{1})-\mathcal{T}(\nu_{2},\mu_{2})\|_{(\mathcal{B}_{\alpha T})^{d}}
\\
\leq \left\|e^{\Delta(T-t)}\nabla\left[G(\Omega(\nu_{1},\mu_{1}),\cdot)-G(\Omega(\nu_{2},\mu_{2}),\cdot)\right]\right\|_{(\mathcal{B}_{\alpha T})^{d}}
\\
+\|I^{-}(H_{1}(\nu_{1},\mu_{1})-H_{1}(\nu_{2},\mu_{2}))\|_{(\mathcal{B}_{\alpha T})^{d}}.
\end{multline}

For the first term on the right-hand side of \eqref{T1Contracting}, we begin with the elementary inequalities
\begin{multline}\nonumber
\left\|e^{\Delta(T-t)}\nabla\left[G(\Omega(\nu_{1},\mu_{1}),\cdot)-G(\Omega(\nu_{2},\mu_{2}),\cdot)\right]\right\|_{(\mathcal{B}_{\alpha T})^{d}}
\\
\leq \left\|\nabla\left[G(\Omega(\nu_{1},\mu_{1}),\cdot)-G(\Omega(\nu_{2},\mu_{2}),\cdot)\right]\right\|_{(B_{\alpha T})^{d}}
\\
\leq \|G(\Omega(\nu_{1},\mu_{1}),\cdot)-G(\Omega(\nu_{2},\mu_{2}),\cdot)\|_{B_{1,\alpha T}}
\\
\leq \tilde{c}_{G}\|\Omega(\nu_{1},\mu_{1})-\Omega(\nu_{2},\mu_{2})\|_{PM^{\alpha T}}
\Psi_{2}(\|\Omega(\nu_{1},\mu_{1})\|_{PM^{\alpha T}}+\|\Omega(\nu_{2},\mu_{2})\|_{PM^{\alpha T}}).
\end{multline}
Of course, we have used the Lipschitz mapping assumption on $G$ from Assumption \ref{payoffAssumption}.
Analagously to \eqref{UpsilonDefinition}, we may define a quantity $\tilde{\Upsilon}$ such that
\begin{equation}\nonumber
\Psi_{2}(\|\Omega(\nu_{1},\mu_{1})\|_{PM^{0}}+\|\Omega(\nu_{2},\mu_{2})\|_{PM^{0}})\leq \tilde{\Upsilon}.
\end{equation}
To be definite, we give the definition of $\tilde{\Upsilon},$ which is
\begin{equation}\label{UpsilonDefinition}
\tilde{\Upsilon}=
\Psi_{2}\left(2R_{0}+\frac{2}{1-\alpha}(R_{0}+c_{G}R_{0}\Psi_{1}(R_{0}))^{2}c_{f_{2}}c_{g_{2}}(2c_{G}R_{0}\Psi_{1}(R_{0}))^{p_{f_{2}}+p_{g_{2}}}
\right).
\end{equation}
We next use \eqref{omegaDefinition}, the definition of $\Omega,$ noticing that $e^{\Delta T}m_{0}$ cancels.  We then also use
other elementary estimates to arrive at
\begin{multline}\nonumber
\left\|e^{\Delta(T-t)}\nabla\left[G(\Omega(\nu_{1},\mu_{1}),\cdot)-G(\Omega(\nu_{2},\mu_{2}),\cdot)\right]\right\|_{(\mathcal{B}_{\alpha T})^{d}}
\\
\leq \tilde{c}_{G}\tilde{\Upsilon}\|I^{+}(\mu_{1},H_{2}(\nu_{1},\mu_{1}))(\cdot,T)-I^{+}(\mu_{2},H_{2}(\nu_{2},\mu_{2}))(\cdot,T)\|_{PM^{\alpha T}}
\\
\leq \tilde{c}_{G}\tilde{\Upsilon}\|I^{+}(\mu_{1}-\mu_{2},H_{2}(\nu_{1},\mu_{1}))\|_{\mathcal{PM}^{\alpha}}
\\
+ \tilde{c}_{G}\tilde{\Upsilon}\|I^{+}(\mu_{2},H_{2}(\nu_{1},\mu_{1})-H_{2}(\nu_{2},\mu_{2}))\|_{\mathcal{PM}^{\alpha}}.
\end{multline}
We next use the bound for $I^{+}$ from Lemma \ref{I+Lemma} in Appendix \ref{appendix}, finding
\begin{multline}\nonumber
\left\|e^{\Delta(T-t)}\nabla\left[G(\Omega(\nu_{1},\mu_{1}),\cdot)-G(\Omega(\nu_{2},\mu_{2}),\cdot)\right]\right\|_{(\mathcal{B}_{\alpha T})^{d}}
\\
\leq \tilde{c}_{G}\tilde{\Upsilon}\frac{1}{1-\alpha}\|\mu_{1}-\mu_{2}\|_{\mathcal{PM}^{\alpha}}\|H_{2}(\nu_{1},\mu_{1})\|_{(\mathcal{B}_{\alpha T})^{d}}
\\
+
\tilde{c}_{G}\tilde{\Upsilon}
\frac{1}{1-\alpha}\|\mu_{2}\|_{\mathcal{PM}^{\alpha}}\|H_{2}(\nu_{1},\mu_{1})-H_{2}(\nu_{2},\mu_{2})\|_{(\mathcal{B}_{\alpha T})^{d}}.
\end{multline}
For the first term on the right-hand side we use \eqref{H2Bound}, and for the second term on the right-hand side we use Lemma \ref{hamiltonianLipschitz}.
We also use $\|\nu_{i}\|_{(\mathcal{B}_{\alpha T})^{d}}\leq 2R_{1},$ 
and $\|\mu_{i}\|_{\mathcal{PM}^{\alpha}}\leq R_{0}+R_{1}.$ These considerations yield
the bound
\begin{multline}\label{LipschitzBoundPayoff}
\left\|e^{\Delta(T-t)}\nabla\left[G(\Omega(\nu_{1},\mu_{1}),\cdot)-G(\Omega(\nu_{2},\mu_{2}),\cdot)\right]\right\|_{(\mathcal{B}_{\alpha T})^{d}}
\\
\leq \tilde{c}_{G}\tilde{\Upsilon}\frac{2}{1-\alpha}(R_{0}+R_{1})
\Bigg(c_{f_{2}}c_{g_{2}}(2R_{1})^{p_{f_{2}}+p_{g_{2}}}
+c_{f_{2}}\tilde{c}_{g_{2}}(2R_{1})^{p_{f_{2}}+\tilde{p}_{g_{2}}}(R_{0}+R_{1})
\\
+\tilde{c}_{f_{2}}c_{g_{2}}(2R_{1})^{\tilde{p}_{f_{2}}+p_{g_{2}}}(R_{0}+R_{1})
\Bigg)
\left(\|\mu_{1}-\mu_{2}\|_{\mathcal{PM}^{\alpha}}+\|\nu_{1}-\nu_{2}\|_{(\mathcal{B}_{\alpha T})^{d}}
\right).
\end{multline}
We see that the constant in front of $\|\mu_{1}-\mu_{2}\|_{\mathcal{PM}^{\alpha}}+\|\nu_{1}-\nu_{2}\|_{(\mathcal{B}_{\alpha T})^{d}}$ can be made small
in either case (a) or case (b) of the theorem.

For the second term on the right-hand side of \eqref{T1Contracting}, we begin to estimate it as
\begin{equation}\nonumber
\|I^{-}(H_{1}(\nu_{1},\mu_{1})-H_{1}(\nu_{2},\mu_{2}))\|_{(\mathcal{B}_{\alpha T})^{d}}
\leq d \|H_{1}(\nu_{1},\mu_{1})-H_{1}(\nu_{2},\mu_{2})\|_{\mathcal{B}_{\alpha T}}.
\end{equation}
We can then use Lemma \ref{hamiltonianLipschitz}.  We find the following:
\begin{multline}\label{LipschitzBoundI-}
\|I^{-}(H_{1}(\nu_{1},\mu_{1})-H_{1}(\nu_{2},\mu_{2}))\|_{(\mathcal{B}_{\alpha T})^{d}}
\\
\leq
d \Bigg(2c_{f_{1}}\tilde{c}_{g_{1}}(2R_{1})^{p_{f_{1}}+\tilde{p}_{g_{1}}}(R_{0}+R_{1})
+2c_{g_{1}}\tilde{c}_{f_{1}}(2R_{1})^{p_{g_{1}}+\tilde{p}_{f_{1}}}(R_{0}+R_{1})
\\
+c_{f_{1}}c_{g_{1}}(2R_{1})^{p_{f_{1}}+p_{g_{1}}}\Bigg)
\left(\|\nu_{1}-\nu_{2}\|_{(\mathcal{B}_{\alpha T})^{d}}+\|\mu_{1}-\mu_{2}\|_{\mathcal{PM}^{\alpha}}\right),
\end{multline}
In case (b), we can take $c_{f_{1}}\tilde{c}_{g_{1}},$ $c_{g_{1}}\tilde{c}_{f_{1}},$ and $c_{f_{1}}c_{g_{1}}$ sufficiently small.
In case (a), we have positive powers of $R_{1}$ present on the right-hand side, and
we can make $R_{1}$ small by taking $c_{G}$ sufficiently small.  Thus, in either case, we may take the
constant in front of $\|\mu_{1}-\mu_{2}\|_{\mathcal{PM}^{\alpha}}+\|\nu_{1}-\nu_{2}\|_{(\mathcal{B}_{\alpha T})^{d}}$ small.

We next turn to the contraction estimates for $\mathcal{T}_{2}.$  We begin with
\begin{equation}\label{T2DifferenceIsI+}
\|\mathcal{T}_{2}(\nu_{1},\mu_{1})-\mathcal{T}_{2}(\nu_{2},\mu_{2})\|_{\mathcal{PM}^{\alpha}}
=\|I^{+}(\mu_{1}H_{2}(\nu_{1},\mu_{1})-I^{+}(\mu_{2}H_{2}(\nu_{2},\mu_{2}))\|_{\mathcal{PM}^{\alpha}}.
\end{equation}
We add and subtract once and use the triangle inequality to find
\begin{multline}\nonumber
\|\mathcal{T}_{2}(\nu_{1},\mu_{1})-\mathcal{T}_{2}(\nu_{2},\mu_{2})\|_{\mathcal{PM}^{\alpha}}
\\
\leq
\|I^{+}(\mu_{1}H_{2}(\nu_{1},\mu_{1}))-I^{+}(\mu_{2}H_{2}(\nu_{1},\mu_{1}))\|_{\mathcal{PM}^{\alpha}}
\\
+
\|I^{+}(\mu_{2}H_{2}(\nu_{1},\mu_{1}))-I^{+}(\mu_{2}H_{2}(\nu_{2},\mu_{2}))\|_{\mathcal{PM}^{\alpha}}.
\end{multline}
Using the bound for $I^{+}$ together with the product estimate \eqref{productEstimate}, we may bound this as
\begin{multline}\nonumber
\|\mathcal{T}_{2}(\nu_{1},\mu_{1})-\mathcal{T}_{2}(\nu_{2},\mu_{2})\|_{\mathcal{PM}^{\alpha}}
\leq \frac{1}{1-\alpha}
\|\mu_{1}-\mu_{2}\|_{\mathcal{PM}^{\alpha}}\|H_{2}(\nu_{1},\mu_{1})\|_{(\mathcal{B}_{\alpha T})^{d}}
\\
+\frac{1}{1-\alpha}
\|\mu_{2}\|_{\mathcal{PM}^{\alpha}}\|H_{2}(\nu_{1},\mu_{1})-H_{2}(\nu_{2},\mu_{2})\|_{(\mathcal{B}_{\alpha T})^{d}}.
\end{multline}
For the first term on the right-hand side we use \eqref{H2Bound} to bound $H_{2},$ and for the second term on the right-hand
side we use the Lipschitz estimate for $H_{2}$ from Lemma \ref{hamiltonianLipschitz}.  This yields the estimate
\begin{multline}\label{LipschitzBoundI+}
\|\mathcal{T}_{2}(\nu_{1},\mu_{1})-\mathcal{T}_{2}(\nu_{2},\mu_{2})\|_{\mathcal{PM}^{\alpha}}
\\
\leq \frac{2}{1-\alpha}(R_{0}+R_{1})\Bigg(c_{f_{2}}c_{g_{2}}(2R_{1})^{p_{f_{2}}+p_{g_{2}}}
+c_{f_{2}}\tilde{c}_{g_{2}}(2R_{1})^{p_{f_{2}}+\tilde{p}_{g_{2}}}(R_{0}+R_{1})
\\
+\tilde{c}_{f_{2}}c_{g_{2}}(2R_{1})^{\tilde{p}_{f_{2}}+p_{g_{2}}}(R_{0}+R_{1})\Bigg)
\left(\|\mu_{1}-\mu_{2}\|_{\mathcal{PM}^{\alpha}}+\|\nu_{1}-\nu_{2}\|_{(\mathcal{B}_{\alpha T})^{d}}\right).
\end{multline}
Again, in case (a) we have positive powers of $R_{1}$ on the right-hand side, and $R_{1}$ can be made small by taking
$c_{G}$ small.  In either case, we see that we may
make the constant in front of $\|\mu_{1}-\mu_{2}\|_{\mathcal{PM}^{\alpha}}+\|\nu_{1}-\nu_{2}\|_{(\mathcal{B}_{\alpha T})^{d}}$ small.

This completes the proof of the theorem.
\end{proof}

\section{Continuous dependence and weak-$*$ convergence of measures}\label{convergenceSection}
In this section we prove two results on continuous dependence of the solutions on the initial data, $m_{0}.$  The first of these
completes well-posedness of our mean field games system with data in $PM^{0},$ in that it shows if $m_{0}$ and $\tilde{m}_{0}$ are close
in $PM^{0},$ then the corresponding solutions $(v,m)$ and $(\tilde{v},\tilde{m})$ remain close in $(\mathcal{B}_{\alpha T})^{d}\times\mathcal{PM}^{\alpha}.$
However, with the bigger picture of mean field games as the limit of $N$-player games as $N$ goes to infinity in mind, this is not fully satisfactory.
The appropriate data when considering $N$-player games is a sum of Dirac masses.  We wish to have a continuous dependence result that shows
that if there is a small change in the Dirac masses, then there is a small change in the solution of the mean field games system.  However, consider
in one dimension a Dirac mass supported at $x=0$ and another Dirac mass supported at $x=\epsilon>0;$ even with $\epsilon>0$ small, if $\epsilon$ is irrational,
then we have
\begin{equation}\nonumber
\|\delta_{0}-\delta_{\epsilon}\|_{PM^{0}}=\sup_{k\in\mathbb{Z}}|1-e^{ik\epsilon}|=2,
\end{equation}
since there will be values of $k$ for which $k\epsilon$ is arbitrarily close to odd multiples of $\pi.$  However, $\delta_{\epsilon}$ does converge to
$\delta_{0}$ in the sense of weak-$*$ convergence of bounded measures.  Thus we will also consider sequences of data converging in this sense, and show
convergence of the corresponding solutions $(v,m).$

We first state and prove our continuous dependence theorem with respect to the pseudomeasure norms.  Note that this also gives a stronger uniqueness
result than in Theorem \ref{Main1}.  We will use this version of the uniqueness result in proving Theorem \ref{Main3} below.

\begin{theorem} \label{Main2}
Let $m_{0}^{1}\in PM^{0}$ and $m_{0}^{2}\in PM^{0}$ be given.  Let $(v_{1},m_{1})\in(\mathcal{B}_{\alpha T})^{d}\times\mathcal{PM}^{\alpha}$ and
$(v_{2},m_{2})\in(\mathcal{B}_{\alpha T})^{d}\times\mathcal{PM}^{\alpha}$ be solutions of \eqref{duhamelM}, \eqref{duhamelV} with this data, respectively.
Let $M>0$ be such that $\|m_{i}\|_{\mathcal{PM}^{\alpha}}\leq M,$ for $i\in\{1,2\}.$

(Case 1) Assume
\begin{equation}\nonumber
p_{f_{1}}+p_{g_{1}}>0,\qquad p_{f_{1}}+\tilde{p}_{g_{1}}>0,\qquad p_{f_{1}}+\tilde{p}_{g_{1}}>0,
\end{equation}
\begin{equation}\nonumber
p_{f_{2}}+p_{g_{2}}>0,\qquad p_{f_{2}}+\tilde{p}_{g_{2}}>0,\qquad p_{f_{2}}+\tilde{p}_{g_{2}}>0.
\end{equation}
There exists $\epsilon>0$ (depending on $M$ and on the constants associated to the bounds for $H_{1}$ and $H_{2}$, and on $d$ and $\alpha$)
such that if $\|v_{i}\|_{(\mathcal{B}_{\alpha T})^{d}}<\epsilon$ for $i\in\{1,2\},$ then
\begin{equation}\nonumber
\|v_{1}-v_{2}\|_{(\mathcal{B}_{\alpha T})^{d}}+\|m_{1}-m_{2}\|_{\mathcal{PM}^{\alpha}}\leq 4\|m_{0}^{1}-m_{0}^{2}\|_{PM^{0}}.
\end{equation}

(Case 2) Assume that $\|v_{1}\|_{(\mathcal{B}_{\alpha T})^{d}}\leq M$ and $\|v_{2}\|_{(\mathcal{B}_{\alpha T})^{d}}\leq M.$  There exists $\delta>0$
(depending on $M$ and on the constants associated to the bounds for $H_{1}$ and $H_{2}$, and on $d$ and $\alpha$)  such that if
\begin{equation}\nonumber
c_{f_{1}}c_{g_{1}}+\tilde{c}_{f_{1}}c_{g_{1}}+c_{f_{1}}\tilde{c}_{g_{1}}
+c_{f_{2}}c_{g_{2}}+\tilde{c}_{f_{2}}c_{g_{2}}+c_{f_{2}}\tilde{c}_{g_{2}}<\delta,
\end{equation}
then
\begin{equation}\nonumber
\|v_{1}-v_{2}\|_{(\mathcal{B}_{\alpha T})^{d}}+\|m_{1}-m_{2}\|_{\mathcal{PM}^{\alpha}}\leq 4\|m_{0}^{1}-m_{0}^{2}\|_{PM^{0}}.
\end{equation}
\end{theorem}

\begin{proof}
We consider the norm of the difference $v_{1}-v_{2}$ in $(\mathcal{B}_{\alpha T})^{d}.$ We use the Duhamel formula \eqref{duhamelV} and the triangle inequality
so that we have
\begin{multline}\nonumber
\|v_{1}-v_{2}\|_{(\mathcal{B}_{\alpha T})^{d}}\leq \| e^{\Delta(T-t)}\nabla (G(\Omega(v_{1},m_{1}))-G(\Omega(v_{2},m_{2})))\|_{(\mathcal{B}_{\alpha T})^{d}}
\\
+\|I^{-}(H_{1}(v_{1},m_{1})-H_{1}(v_{2},m_{2}))\|_{(\mathcal{B}_{\alpha T})^{d}}.
\end{multline}
We have already considered each of the terms on the right-hand side while establishing the contracting property of $\mathcal{T}.$
For the first term on the right-hand side, considering \eqref{LipschitzBoundPayoff}, we can take either $\epsilon>0$ small enough (in Case 1) or $\delta>0$ small
enough (in case 2) so that
\begin{multline}\nonumber
\| e^{\Delta(T-t)}\nabla (G(\Omega(v_{1},m_{1}))-G(\Omega(v_{2},m_{2})))\|_{(\mathcal{B}_{\alpha T})^{d}}
\\
\leq \frac{1}{4}\left(\|m_{1}-m_{2}\|_{\mathcal{PM}^{\alpha}}+\|v_{1}-v_{2}\|_{(\mathcal{B}_{\alpha T})^{d}}\right).
\end{multline}
Similarly, in light of \eqref{LipschitzBoundI-}, we see that we can take either $\epsilon>0$ small enough (in Case 1) or $\delta>0$ small
enough (in case 2) so that
\begin{multline}\nonumber
\|I^{-}(H_{1}(v_{1},m_{1})-H_{1}(v_{2},m_{2}))\|_{(\mathcal{B}_{\alpha T})^{d}}
\\
\leq \frac{1}{4}\left(\|m_{1}-m_{2}\|_{\mathcal{PM}^{\alpha}}+\|v_{1}-v_{2}\|_{(\mathcal{B}_{\alpha T})^{d}}\right).
\end{multline}

Turning to $m,$ using the Duhamel formula \eqref{duhamelM} and the triangle inequality, we have
\begin{multline}\nonumber
\|m_{1}-m_{2}\|_{\mathcal{PM}^{\alpha}} \leq \left\|e^{\Delta t}(m_{0}^{1}-m_{0}^{2})\right\|_{\mathcal{PM}^{\alpha}}
\\
+\|I^{+}(m_{1},H_{2}(v_{1},m_{1}))-I^{+}(m_{2},H_{2}(v_{2},m_{2}))\|_{\mathcal{PM}^{\alpha}}.
\end{multline}

The first term on the right-hand side can be bounded simply as
\begin{equation}\nonumber
\left\|e^{\Delta t}(m_{0}^{1}-m_{0}^{2})\right\|_{\mathcal{PM}^{\alpha}} \leq \|m_{0}^{1}-m_{0}^{2}\|_{PM^{0}}.
\end{equation}
The second term on the right-hand side again has been considered previously; considering \eqref{T2DifferenceIsI+} and \eqref{LipschitzBoundI+}, we again
may take either $\epsilon>0$ small enough (in Case 1) or $\delta>0$ small
enough (in case 2) so that
\begin{multline}\nonumber
\|I^{+}(m_{1},H_{2}(v_{1},m_{1}))-I^{+}(m_{2},H_{2}(v_{2},m_{2}))\|_{\mathcal{PM}^{\alpha}}
\\
\leq
\frac{1}{4}\left(\|m_{1}-m_{2}\|_{\mathcal{PM}^{\alpha}}+\|v_{1}-v_{2}\|_{(\mathcal{B}_{\alpha T})^{d}}\right).
\end{multline}

Combining all of these calculations, we arrive at
\begin{equation}\nonumber
\|v_{1}-v_{2}\|_{(\mathcal{B}_{\alpha T})^{d}}+\|m_{1}-m_{2}\|_{\mathcal{PM}^{\alpha}}\leq 4\|m_{0}^{1}-m_{0}^{2}\|_{PM^{0}}.
\end{equation}
This completes the proof of the theorem.
\end{proof}

Finally, we consider the continuous dependence problem with respect to weak-$*$ convergence of measures. Let $\mathcal{BM}(\mathbb{T}^d)$ denote the space of bounded Radon measures on $\mathbb{T}^d$.

\begin{theorem} \label{Main3}
Let $\alpha\in(0,1)$.  Let $\{m^{n}_{0}\} \subset \mathcal{BM}(\mathbb{T}^d)$ and assume that, for some $m_0 \in \mathcal{BM}(\mathbb{T}^d)$, we have  $m_0^n$ converges weak-$*$-$\mathcal{BM}(\mathbb{T}^d)$ to $m_{0}$.
Let $(v^{n},m^{n})$,  $(v^{\infty},m^{\infty})$ be the corresponding solutions  of \eqref{duhamelM}, \eqref{duhamelV} with initial data $m_0^n$,  $m_0$ respectively. Assume these solutions to be uniformly bounded in $(\mathcal{B}_{\alpha T})^{d}\times\mathcal{PM}^{\alpha}.$
Assume that one of the smallness conditions of Theorem \ref{Main2} associated to this uniform bound is satisfied.
Then the sequence $v^{n}$ converges uniformly to $v^{\infty}$ on $\mathbb{T}^{d}\times[0,T],$
and, at every time $t\in[0,T],$ $m^{n}(\cdot,t)$ converges weak-$*$-$\mathcal{BM}(\mathbb{T}^d)$ to $m^{\infty}(\cdot,t)$.
\end{theorem}

\begin{proof}

At every time $t\in[0,T],$ all of the functions $v^{n}(\cdot,t)$ are analytic with radius of analyticity at least $\alpha T.$
Differentiating \eqref{duhamelV} with respect to time, we see that the $v^{n}$ are in fact strong solutions, i.e. the $v^{n}$
all satisfy \eqref{finalVEquation}.  The uniform bound in $\mathcal{B}_{\alpha T}$ implies that $\nabla v^n$ and $v_{t}^n$
are all bounded, and thus $v^{n}$ is an equicontinuous family on $\mathbb{T}^{d}\times[0,T].$  The Arzela-Ascoli theorem
implies the existence of a subsequence (which we do not relabel) which converges uniformly to a continuous function, $\tilde{v}^{\infty},$
on $[0,T] \times \mathbb{T}^{d}.$
Since the uniform limit of analytic functions is analytic, we can further conclude analyticity of $\tilde{v}^{\infty}$ at every time.

For the sequence $m^{n},$ we consider convergence at each time.  At $t=0,$ then, by assumption, we have
that $m^{n}(\cdot,0)=m^{n}_{0}$ converges weak-$*$-$\mathcal{BM}(\mathbb{T}^d)$ to $m^{\infty}(\cdot,0)\equiv m_{0}.$  Now let $t>0$ be given.
Since the solutions $m^{n}$ and $m^{\infty}$ are all uniformly bounded  in $\mathcal{PM}^{\alpha},$ we have that
\begin{equation}\nonumber
\sup_{k\in\mathbb{Z}^{d}}e^{\alpha t|k|}\left|\widehat{m^{j}}(k,t)\right|
\leq
\sup_{\tau\in[0,T]}\sup_{k\in\mathbb{Z}^{d}}e^{\alpha \tau|k|}\left|\widehat{m^{j}}(k,\tau)\right|
=\|m^{j}\|_{\mathcal{PM}^{\alpha}}
\leq{\bar{K}},
\end{equation}
for $j\in\mathbb{N}\cup\{\infty\},$ and for some $\bar{K}>0.$
Thus, we have that, for all such $j,$ $k,$ and for our specified $t\in(0,T],$
\begin{equation}\nonumber
\left|\widehat{m^{j}}(k,t)\right|\leq \bar{K}e^{-\alpha t|k|}.
\end{equation}
This implies that the infinity norms of all of the $m^{j}(\cdot,t)$ are uniformly bounded:
\begin{equation}\nonumber
\|m^{j}(\cdot,t)\|_{L^{\infty}(\mathbb{T}^{d})}=\sup_{x\in\mathbb{T}^{d}}|m^{j}(x,t)|
\leq \sum_{k\in\mathbb{Z}^{d}}\left|\widehat{m^{j}}(k,t)\right|
\leq \bar{K}\sum_{k\in\mathbb{Z}^{d}}e^{-\alpha t |k|}.
\end{equation}
For our fixed $t>0,$ this is a convergent sum, so we conclude that the $L^{\infty}$-norms are uniformly bounded.
(Note that this bound depends badly on $t,$ and hence we do not take $t$ to zero with this bound.)
This implies, in particular, that for our fixed $t>0,$ the total measure of all the $m^{j}(\cdot,t)$ is uniformly bounded.

We wish to show that $m^{n}(\cdot,t)$ converges weak-$*$-$\mathcal{BM}(\mathbb{T}^d)$  to $m^{\infty}(\cdot,t).$  This means that we
wish to show that the pairing $\langle m^{n}(\cdot,t),\phi\rangle$ converges to $\langle m^{\infty}(\cdot,t),\phi\rangle$ for every function $\phi\in C(\mathbb{T}^{d}).$
Since the measures $m^{n}(\cdot,t)$ are uniformly bounded, we claim it is sufficient to show the convergence on a dense subset.  In particular, we take
$C^{\infty}(\mathbb{T}^{d})$ to be this dense subset.  Indeed, assume that for all $\tilde{\phi}\in\mathbb{C}^{\infty}(\mathbb{T}^{d})$ we have
\begin{equation}\nonumber
\langle m^{n}(\cdot,t)-m^{\infty}(\cdot,t),\tilde{\phi}\rangle \rightarrow0,\  \mathrm{as}\  n\rightarrow\infty.
\end{equation}
Then we write
\begin{multline}\nonumber
\langle m^{n}(\cdot,t)-m^{\infty}(\cdot,t),\phi\rangle
=
\langle m^{n}(\cdot,t),\phi-\tilde{\phi}\rangle
+
\langle m^{n}(\cdot,t)-m^{\infty}(\cdot,t),\tilde{\phi}\rangle
\\
+
\langle m^{\infty}(\cdot,t),\tilde{\phi}-\phi\rangle.
\end{multline}
Since $m^{n}(\cdot,t)$ is a uniformly bounded sequence and $\tilde{\phi}$ can be taken to be as close as desired to $\phi$ in $C(\mathbb{T}^{d}),$
the first term on the right-hand side can be made arbitrarily small.  Similarly, the third term on the right-hand side may be made arbitrarily small.
For the second term on the right-hand side, our assumption is that this goes to zero for $\tilde{\phi}\in C^{\infty}(\mathbb{T}^{d}).$  Thus, we have
proved our claim that it is sufficient to consider a dense subset of $C(\mathbb{T}^{d}).$  

Let $\tilde{\phi}\in C^{\infty}(\mathbb{T}^{d})$ be given.

Let $\sigma$ be such that
$0<\sigma<T.$  Then for all $t\in[\sigma, T],$ the functions $m^{n}(\cdot,t)$ are all bounded and analytic with radius of analyticity at least $\alpha t,$
which is at least $\alpha\sigma.$  As was the case with $v^n$, we may differentiate \eqref{duhamelM} and obtain that $m^n_t$ and $\nabla m^n$ are all uniformly bounded on $\mathbb{T}^{d}\times[\sigma,T]$, hence $m^n$ are, also, an equicontinuous family on $ \mathbb{T}^{d}\times[\sigma,T]$. Using, again, the Arzela-Ascoli theorem, this implies uniform convergence of a subsequence (which we do not relabel) to a limit $\tilde{m}^{\infty}$
with temporal domain $[\sigma,T].$  We may then perform a diagonal argument to find convergence to $\tilde{m}^{\infty}$ on $(0,T].$
This implies that at our specified time $t\in(0,T],$ the sequence $m^{n}(\cdot,t)$ converges uniformly to $\tilde{m}^{\infty}(\cdot,t).$
Therefore, 
\begin{equation}\nonumber
\langle m^{n}(\cdot,t)-\tilde{m}^{\infty}(\cdot,t),\tilde{\phi}\rangle \rightarrow 0,\ \mathrm{as}\ n\rightarrow\infty.
\end{equation}
Since thus far $\tilde{m}^{\infty}(\cdot,0)$ is undefined, we now define $\tilde{m}^{\infty}(\cdot,0)=m_{0}.$ Note that, by construction, $\tilde{m}^\infty$ is an analytic function on $\mathbb{T}^d\times(0,T].$

What remains is to identify the limit $\tilde{m}^{\infty}$ as actually being $m^{\infty}$ at positive times.  We know that $m^{n}(\cdot,t)$ converges uniformly, thus weak-$*$-$\mathcal{BM}(\mathbb{T}^d)$, to 
$\tilde{m}^{\infty}(\cdot,t).$ If we can conclude that $m^{n}(\cdot,t)$ converges weak-$*$-$\mathcal{BM}(\mathbb{T}^d)$ to something else, then this something else must be equal to $\tilde{m}^{\infty}(\cdot,t).$

The equation satisfied by $m^{n}$ is
\begin{equation}\label{mnDuhamel}
m^{n}(\cdot,t)=e^{\Delta t}m_{0}^{n}+\int_{0}^{t}e^{\Delta(t-s)}\mathrm{div}\left(
m^{n}g_{2}(v^{n})\int_{\mathbb{T}^{d}}f_{2}(v^{n})m^{n}\ \mathrm{d}x\right)\ \mathrm{d}s,
\end{equation}
so we can express $\langle m^{n}(\cdot,t),\tilde{\phi}\rangle$ as 
\begin{multline}\nonumber
\langle m^{n}(\cdot,t),\tilde{\phi}\rangle
=
\langle e^{\Delta t}m^{n}_{0}, \tilde{\phi}\rangle
\\
+\left\langle \int_{0}^{t}e^{\Delta(t-s)}\mathrm{div}\left(
m^{n}g_{2}(v^{n})\int_{\mathbb{T}^{d}}f_{2}(v^{n})m^{n}\ \mathrm{d}x\right)\ \mathrm{d}s,
\tilde{\phi}\right\rangle.
\end{multline}
We wish to pass to the limit in the right-hand-side of this identity. To do so we will observe that there are several weak-strong pairs at play.

We will show that the right-hand-side above converges to
\begin{equation}\nonumber
\langle e^{\Delta t}m_{0}, \tilde{\phi}\rangle+\left\langle \int_{0}^{t}e^{\Delta(t-s)}\mathrm{div}\left(
\tilde{m}^{\infty}g_{2}(\tilde{v}^{\infty})\int_{\mathbb{T}^{d}}f_{2}(\tilde{v}^{\infty})\tilde{m}^{\infty}\ \mathrm{d}x\right)\ \mathrm{d}s,
\tilde{\phi}\right\rangle.
\end{equation}
This will imply that $\tilde{m}^{\infty}$ satisfies
\begin{equation}\label{mDuhamelLimit}
\tilde{m}^{\infty}(\cdot,t)=e^{\Delta t}m_{0}+\int_{0}^{t}e^{\Delta(t-s)}\mathrm{div}\left(
\tilde{m}^{\infty}g_{2}(\tilde{v}^{\infty})\int_{\mathbb{T}^{d}}f_{2}(\tilde{v}^{\infty})\tilde{m}^{\infty}\ \mathrm{d}x\right)\ \mathrm{d}s.
\end{equation}
After establishing the corresponding formula for $\tilde{v}^{\infty},$ this will imply, by uniqueness of the solution, that $\tilde{m}^{\infty}=m^{\infty}$ and $\tilde{v}^{\infty}=v^{\infty}.$ Moreover, we do not need to pass to further subsequences. 

It is clear that $\langle e^{\Delta t}m^{n}_{0},\tilde{\phi}\rangle$ converges to $\langle e^{\Delta t}m_{0},\tilde{\phi}\rangle$ as $n$ goes
to infinity, for this fixed $t>0,$ since $m^{n}_{0}$ converges weak-$*$ in the sense of bounded measures on $\mathbb{T}^{d}.$  
We therefore focus on the convergence of the nonlinear term, that is, the pairing of the Duhamel integral with $\tilde{\phi}.$  
We can write this pairing as
\begin{multline}\nonumber
\left\langle \int_{0}^{t}e^{\Delta(t-s)}\mathrm{div}\left(
m^{n}g_{2}(v^{n})\int_{\mathbb{T}^{d}}f_{2}(v^{n})m^{n}\ \mathrm{d}x\right)\ \mathrm{d}s,
\tilde{\phi}\right\rangle
=
\\
\int_{\mathbb{T}^{d}}\int_{0}^{t}\Bigg[\tilde{\phi}(y)\cdot
\\
e^{\Delta(t-s)}\mathrm{div}_{y}\left(m^{n}(y,s)g_{2}(v^{n}(y,s))\int_{\mathbb{T}^{d}}f_{2}(v^{n}(x,s))m^{n}(x,s)
\ \mathrm{d}x\right)\Bigg]\ \mathrm{d}s\mathrm{d}y.
\end{multline}
By Parseval's Theorem, we can move the heat semigroup to act instead on $\tilde{\phi};$ we also integrate by parts to move the $\mathrm{div}_{y}$ operator.
These considerations yield
\begin{multline}\nonumber
\left\langle \int_{0}^{t}e^{\Delta(t-s)}\mathrm{div}\left(
m^{n}g_{2}(v^{n})\int_{\mathbb{T}^{d}}f_{2}(v^{n})m^{n}\ \mathrm{d}x\right)\ \mathrm{d}s,
\tilde{\phi}\right\rangle
=
\\
-\int_{\mathbb{T}^{d}}\int_{0}^{t}\Bigg[\nabla_{y}(e^{\Delta(t-s)}\tilde{\phi}(y))\cdot
\\
\left(m^{n}(y,s)g_{2}(v^{n}(y,s))\int_{\mathbb{T}^{d}}f_{2}(v^{n}(x,s))m^{n}(x,s)
\ \mathrm{d}x\right)\Bigg]\ \mathrm{d}s\mathrm{d}y.
\end{multline}
We will be showing that this converges to
\begin{multline}\nonumber
-\int_{\mathbb{T}^{d}}\int_{0}^{t}\Bigg[\nabla_{y}(e^{\Delta(t-s)}\tilde{\phi}(y))\cdot
\\
\left(\tilde{m}^{\infty}(y,s)g_{2}(\tilde{v}^{\infty}(y,s))\int_{\mathbb{T}^{d}}f_{2}(\tilde{v}^{\infty}(x,s))\tilde{m}^{\infty}(x,s)
\ \mathrm{d}x\right)\Bigg]\ \mathrm{d}s\mathrm{d}y.
\end{multline}
To do so, we first treat the initial layer, which, as it is integrated in time, is small. Then we consider times bounded away from $0$.

Let $\varepsilon>0$ be given.  We can choose $\sigma^{*}\in(0,T]$ to be sufficiently small so that for all $n,$
\begin{multline}\nonumber
\Bigg|
\int_{\mathbb{T}^{d}}\int_{0}^{\sigma^{*}}\Bigg[\nabla_{y}(e^{\Delta(t-s)}\tilde{\phi}(y))\cdot
\\
\left(m^{n}(y,s)g_{2}(v^{n}(y,s))\int_{\mathbb{T}^{d}}f_{2}(v^{n}(x,s))m^{n}(x,s)
\ \mathrm{d}x\right)\Bigg]\ \mathrm{d}s\mathrm{d}y
\\
-\int_{\mathbb{T}^{d}}\int_{0}^{\sigma^{*}}\Bigg[\nabla_{y}(e^{\Delta(t-s)}\tilde{\phi}(y))\cdot
\\
\left(\tilde{m}^{\infty}(y,s)g_{2}(\tilde{v}^{\infty}(y,s))\int_{\mathbb{T}^{d}}f_{2}(\tilde{v}^{\infty}(x,s))\tilde{m}^{\infty}(x,s)
\ \mathrm{d}x\right)\Bigg]\ \mathrm{d}s\mathrm{d}y
\Bigg| <\frac{\varepsilon}{2}.
\end{multline}
Note that the two individual pieces may each be made smaller than $\varepsilon/4.$
To see this, first change the order of integration, which can be justified by Parseval's relation; in Fourier space the sum is absolutely convergent and bounded uniformy in time. Then use 
the bounds  in $PM^{0},$ pointwise in time,  for $m^{n}$ and $\tilde{m}^{\infty}$, which  follow from the uniform $\mathcal{PM}^{\alpha}$ bound, assumed by hypothesis.
Finally, use  the fact that we are pairing $m^{n}$ and $\tilde{m}^{\infty}$ with uniformly bounded (in $n$ and in $t$) functions in the Wiener algebra.  
As part of this, since $\tilde{\phi}\in C^{\infty}(\mathbb{T}^{d}),$ it is in $H^{s}(\mathbb{T}^{d})$ for all $s\in\mathbb{R},$ and this implies that its Fourier coefficients decay faster than any power of $|k|;$ therefore $\tilde{\phi}$ and its derivatives are in the Wiener algebra.  With the presence of the heat semigroup, this is also true uniformly with respect to $s.$
Thus, with the integrand being uniformly bounded with respect to $s,$ integrating over a small time interval leads to a small result.

Next we consider the times bounded away from $0$. We wish to prove that we can take $n$ large enough so that
\begin{multline}\nonumber
\Bigg|
\int_{\mathbb{T}^{d}}\int_{\sigma^{*}}^{t}\Bigg[\nabla_{y}(e^{\Delta(t-s)}\tilde{\phi}(y))\cdot
\\
\left(m^{n}(y,s)g_{2}(v^{n}(y,s))\int_{\mathbb{T}^{d}}f_{2}(v^{n}(x,s))m^{n}(x,s)
\ \mathrm{d}x\right)\Bigg]\ \mathrm{d}s\mathrm{d}y
\\
-\int_{\mathbb{T}^{d}}\int_{\sigma^{*}}^{t}\Bigg[\nabla_{y}(e^{\Delta(t-s)}\tilde{\phi}(y))\cdot
\\
\left(\tilde{m}^{\infty}(y,s)g_{2}(\tilde{v}^{\infty}(y,s))\int_{\mathbb{T}^{d}}f_{2}(\tilde{v}^{\infty}(x,s))\tilde{m}^{\infty}(x,s)
\ \mathrm{d}x\right)\Bigg]\ \mathrm{d}s\mathrm{d}y
\Bigg| <\frac{\varepsilon}{2}.
\end{multline}
This follows immediately from the fact that we have uniform convergence of $v^{n}$ and $m^{n}$ on $\mathbb{T}^{d}\times[\sigma_{*},T].$

As discussed above, we have now shown that $m^{n}(\cdot,t)$ converges weak-$*$ in the sense of bounded measures on $\mathbb{T}^{d}$
to $\tilde{m}^{\infty}(\cdot,t),$ and also to the right-hand side of \eqref{mDuhamelLimit}.  We conclude that \eqref{mDuhamelLimit} holds.

We similarly must consider convergence in the mild form of the $v$ equation.  The equation satisfied by $v^{n}$ is
\begin{equation}\nonumber
v^{n}(\cdot,t)=e^{\Delta(T-t)}\nabla G(m^{n}(\cdot,T),\cdot)-\int_{t}^{T}e^{\Delta(s-t)}\nabla H_{1}(v^{n},m^{n})\ \mathrm{d}s.
\end{equation}
For $t>0,$ convergence of the Duhamel integral is simpler in this case than the case we have just considered as we have uniform convergence
of both $m^{n}$ and $v^{n}$ for $s\in[t,T].$
Since $m^{n}(\cdot,T)$ converges uniformly to $\tilde{m}^{\infty}(\cdot,T),$ convergence in the first term is straightforward as well.  We therefore can conclude, for $t>0,$
\begin{equation}\label{vLimitEquation}
\tilde{v}^{\infty}(\cdot,t)=e^{\Delta(T-t)}\nabla G(\tilde{m}^{\infty}(\cdot,T),\cdot)-\int_{t}^{T}e^{\Delta(s-t)}\nabla 
H_{1}(\tilde{v}^{\infty},\tilde{m}^{\infty})\ \mathrm{d}s.
\end{equation}
All that remains is to show that this formula holds at time zero, as well.  For this, there is no difficulty in convergence of the Duhamel integral, 
as the dependence on $m^{n}$ is only inside
the integral $\int_{\mathbb{T}^{d}}f_{1}(v^{n})m^{n}\ \mathrm{d}x,$ and this integral is bounded independently of time (see \eqref{ABound} in the appendix).  
Thus \eqref{vLimitEquation} also holds
at time zero.

Since $t>0$ was arbitrary, we conclude that \eqref{mDuhamelLimit} holds for all $t\in (0,T].$ 
Furthermore, we see that \eqref{mDuhamelLimit} also holds in the limit as $t\rightarrow0^{+},$ as the integrand in the Duhamel integral is 
bounded uniformly in time as an element of a negative Sobolev space. Thus we extend the validity of \eqref{mDuhamelLimit} to $[0,T]$. We also have that \eqref{vLimitEquation} holds for all $t\in[0,T].$
However, by Theorem \ref{Main2}, we know the unique solution of \eqref{mDuhamelLimit}, \eqref{vLimitEquation} with initial data $m_{0}$ is $(m^{\infty},v^{\infty}).$
We therefore conclude $(\tilde{m}^{\infty},\tilde{v}^{\infty})=(m^{\infty},v^{\infty}).$  This completes the proof of the theorem.
\end{proof}

\section*{Acknowledgments}
DMA is grateful to the National Science Foundation for support through grant DMS-2307638. AM was partially supported by the US National Science Foundation under grants DMS-1909103, DMS-2206453, DMS-2511023, and  Simons Foundation Grant 1036502. MCLF was supported in part by  
  CNPq through grant 304990/2022-1 and by FAPERJ through grant E-26/201.209/2021. HJNL was partially supported by CNPq through grant 305309/2022-6 and by FAPERJ through grant E-26/201.027/2022.

\appendix

\section{Operator estimates} \label{appendix}

In this appendix we establish the boundedness of the integral operators appearing in the mild formulation
\eqref{duhamelVOriginal}-\eqref{duhamelMOriginal} of equations \eqref{finalVEquation}-\eqref{finalMEquation}, starting with $I^{+}.$

\begin{lemma}\label{I+Lemma}
Let $T>0$ and $\alpha\in\left[0,1\right)$ be given.  The operator $I^{+}$ is a bounded bilinear operator from
$(\mathcal{B}_{\alpha T})^{d}\times\mathcal{PM}^{\alpha}$ to $\mathcal{PM}^{\alpha},$ with operator norm bounded above by
\begin{equation}\nonumber
\|I^{+}\|\leq \frac{1}{1-\alpha}.
\end{equation}
\end{lemma}
\begin{proof}
Computing the $\mathcal{PM}^{\alpha}$ norm of
$I^{+}(f,g),$ we have
\begin{multline}\nonumber
\|I^{+}(f,g)\|_{\mathcal{PM}^{\alpha}}=\sup_{k\in\mathbb{Z}_{*}}\sup_{t\in[0,T]}\Bigg[e^{\alpha t|k|}
\\
\left|\int_{0}^{t}e^{-|k|^{2}(t-s)}\sum_{n=1}^{d}ik_{n}\sum_{j\in\mathbb{Z}}\hat{f}_{n}(k-j,s)\hat{g}(j,s)\ \mathrm{d}s\right|
\Bigg].
\end{multline}
We use the triangle inequality, introduce additional supremums, and adjust factors of exponentials to bound this:
\begin{multline}\nonumber
\|I^{+}(f,g)\|_{\mathcal{PM}^{\alpha}}\leq
\sup_{k\in\mathbb{Z}_{*}^{d}}\sup_{t\in[0,T]}
\Bigg[|k|e^{\alpha t|k|}e^{-|k|^{2}t}
\\
\int_{0}^{t}e^{|k|^{2}s}e^{-\alpha s|k|}\sum_{n=1}^{d}
\left(\sum_{j\in\mathbb{Z}^{d}}\sup_{\tau\in[0,T]}e^{\alpha \tau|k-j|}|\hat{f}_{n}(k-j,\tau)|\right)
\\
\left(\sup_{\ell\in\mathbb{Z}^{d}}\sup_{\tau\in[0,T]}e^{\alpha \tau |\ell|}|\hat{g}(\ell,\tau)|\right)
\ \mathrm{d}s\Bigg].
\end{multline}
We recognize norms, finding the bound
\begin{multline}\label{I+IntegralToBound}
\|I^{+}(f,g)\|_{\mathcal{PM}^{\alpha}}
\\
\leq
\|f\|_{(\mathcal{B}_{\alpha T})^{d}}\|g\|_{\mathcal{PM}^{\alpha}}
\sup_{k\in\mathbb{Z}_{*}^{d}}\sup_{t\in[0,T]}
\left[|k|e^{\alpha t|k|}e^{-|k|^{2}t}\int_{0}^{t}e^{|k|^{2}s}e^{-\alpha s|k|}\ \mathrm{d}s\right].
\end{multline}
We next must evaluate this integral, which is straightforward to do.
We find the bound
\begin{multline}\nonumber
\|I^{+}(f,g)\|_{\mathcal{PM}^{\alpha}} \leq \|f\|_{(\mathcal{B}_{\alpha T})^{d}}\|g\|_{\mathcal{PM}^{\alpha}}\sup_{k\in\mathbb{Z}^{d}_{*}}\sup_{t\in[0,T]}
\left(\frac{1}{|k|-\alpha}-\frac{e^{t|k|(|k|-\alpha)}}{|k|-\alpha}\right)
\\
= \frac{ \|f\|_{(\mathcal{B}_{\alpha T})^{d}}\|g\|_{\mathcal{PM}^{\alpha}}}{1-\alpha}.
\end{multline}
This completes the proof.
\end{proof}

We proceed similarly to bound $I^{-}(h)$ in $(\mathcal{B}_{\alpha T})^{d}.$
\begin{lemma}\label{I-Lemma}
 Let $T>0$ and $\alpha\geq0$ be given.
The operator $I^{-}$ is a bounded linear operator from $\mathcal{B}_{\alpha T}$ to $(\mathcal{B}_{\alpha T})^{d},$
with operator norm bounded above by
\begin{equation}\nonumber
\|I^{-}\|\leq d.
\end{equation}
\end{lemma}
\begin{proof}
We begin by writing
\begin{equation}\nonumber
\|I^{-}(h)\|_{(\mathcal{B}_{\alpha T})^{d}}
=
\sum_{n=1}^{d}\sum_{k\in\mathbb{Z}_{*}^{d}}\sup_{t\in[0,T]} e^{\alpha T|k|}
\left|\int_{t}^{T}e^{-|k|^{2}(s-t)}ik_{n}\hat{h}(k,s)\ \mathrm{d}s\right|.
\end{equation}
We use the triangle inequality, adjust factors of exponentials, and introduce another supremum:
\begin{equation}\nonumber
\|I^{-}(h)\|_{\mathcal{B}_{\alpha T}}
\leq
d\sum_{k\in\mathbb{Z}_{*}^{d}}\sup_{t\in[0,T]}|k|
\int_{t}^{T}e^{-|k|^{2}(s-t)}
\left(\sup_{\tau\in[0,T]}e^{\alpha T|k|}|\hat{h}(k,\tau)|\right)\ \mathrm{d}s.
\end{equation}
Introducing one more supremum, this becomes
\begin{multline}\nonumber
\|I^{-}(h)\|_{(\mathcal{B}_{\alpha T})^{d}}\leq
\\
d\left(\sum_{k\in\mathbb{Z}_{*}^{d}}\sup_{\tau\in[0,T]}e^{\alpha T|k|}
|\hat{h}(k,\tau)|\right)
\left(\sup_{k\in\mathbb{Z}_{*}^{d}}\sup_{t\in[0,T]}|k|
\int_{t}^{T}e^{-|k|^{2}(s-t)}\ \mathrm{d}s\right).
\end{multline}
We now recognize a norm, finding
\begin{equation}\label{I-IntegralToBound}
\|I^{-}(h)\|_{(\mathcal{B}_{\alpha T})^{d}}\leq
d\|h\|_{\mathcal{B}_{\alpha}T}
\left(\sup_{k\in\mathbb{Z}_{*}^{d}}\sup_{t\in[0,T]}|k|
\int_{t}^{T}e^{-|k|^{2}(s-t)}\ \mathrm{d}s\right).
\end{equation}
We compute this integral, finding
\begin{equation}\nonumber
\|I^{-}h\|_{(\mathcal{B}_{\alpha T})^{d}} \leq d\|h\|_{\mathcal{B}_{\alpha T}}\left(\sup_{k\in\mathbb{Z}^{d}_{*}}\sup_{t\in[0,T]}
\frac{1}{|k|}-\frac{e^{-|k|^{2}(T-t)}}{|k|}\right)
\leq d\|h\|_{\mathcal{B}_{\alpha T}}.
\end{equation}
This completes the proof.
\end{proof}

We conclude this section by briefly presenting one more bound.
We introduce the notation $A_{i}=A_{i}(v,m)$ to stand for
\begin{equation}\nonumber
A_{i}=\int_{\mathbb{T}^{d}}f_{i}(v)m\ \mathrm{d}x.
\end{equation}
We note the immediate bound for $A_{i},$ namely
\begin{equation}\label{ABound}
|A_{i}|=\left|\sum_{k\in\mathbb{Z}}\left(\widehat{f_{i}(v)}\right)_{k}\hat{m}_{-k}\right|\leq \|m\|_{\mathcal{PM}^{0}}\|f_{i}(v)\|_{\mathcal{B}_{0}}
\leq \|m\|_{\mathcal{PM}^{\alpha}}\|f_{i}(v)\|_{\mathcal{B}_{\alpha T}}.
\end{equation}

\bibliographystyle{plain}
\bibliography{PM-MFG.bib}

\end{document}